\documentclass[letter,11pt]{article}

\usepackage[linktocpage=true]{hyperref}
\usepackage[margin=1.2 in, top=1 in, bottom= 1.2 in]{geometry}

\usepackage[normalem]{ulem}

\usepackage{lipsum}
\newcommand\blfootnote[1]{%
  \begingroup
  \renewcommand\thefootnote{}\footnote{#1}%
  \addtocounter{footnote}{-1}%
  \endgroup
}

\makeatletter
\g@addto@macro\normalsize{
  \setlength\abovedisplayskip{8pt}
  \setlength\belowdisplayskip{8pt}
  \setlength\abovedisplayshortskip{8pt}
  \setlength\belowdisplayshortskip{8pt}
  }
\makeatother

\usepackage{tocloft}

\usepackage[english]{babel}
\usepackage{amsfonts}
\usepackage{mathrsfs}
\usepackage{bbm}
\usepackage{latexsym}
\usepackage{amssymb}
\usepackage{mathtools}

\usepackage{enumitem}
\setlist{nolistsep}
\usepackage{amsthm}
\usepackage[capitalize]{cleveref}
\crefformat{footnote}{#2\footnotemark[#1]#3} 

\newcommand\eqnitem[1][]{%
  \ifx\relax#1\relax  \item \else \item[#1] \fi
  \abovedisplayskip=0pt\abovedisplayshortskip=0pt~\vspace*{-\baselineskip}}

\newtheoremstyle{plain}{3mm}{3mm}{\slshape}{}{\bfseries}{.}{.5em}{}
\newtheoremstyle{definition}{2mm}{2mm}{}{}{\bfseries}{.}{.5em}{}
\theoremstyle{plain}
	
\newtheorem{theorem}{Theorem}

\newtheorem{lemma}[theorem]{Lemma}

\newtheorem*{clm}{Claim}
\newtheorem{question}[theorem]{Question}
\theoremstyle{definition}
\newtheorem{definition}[theorem]{Definition}

\theoremstyle{plain}
\newcounter{MainTheoremCounter}

\newtheorem{Maintheorem}[MainTheoremCounter]{Theorem}

\theoremstyle{plain}
\newtheorem*{namedthm}{\namedthmname}
\newcounter{namedthm}
	

\usepackage{chngcntr}
\counterwithin{theorem}{section}

\numberwithin{equation}{section}

\allowdisplaybreaks

\usepackage{xcolor}
\definecolor{DarkGreen}{RGB}{1,150,32}
\definecolor{Color2}{rgb}{0.78, 0.11, 0.0}
\hypersetup{citecolor = black,colorlinks,
			linkcolor = black,
			urlcolor = Color2}

\usepackage{titlesec}
\titleformat{\section}
  {\Large\center\bfseries}
  {\thesection.}{.7em}{}
\titlespacing*{\section}{0pt}{3.5ex plus 0ex minus 0ex}{1.5ex plus 0ex}
\titleformat{\subsection}
  {\large\center\bfseries}
  {\thesubsection.}{.7em}{}
\titlespacing*{\subsection}{0pt}{3.5ex plus 0ex minus 0ex}{1.5ex plus 0ex}
\titleformat{\subsubsection}
  {\center\bfseries}
  {\thesubsubsection.}{.7em}{}
\titlespacing*{\subsubsection}{0pt}{3.5ex plus 0ex minus 0ex}{1.5ex plus 0ex}

\addto\captionsenglish{}

\usepackage{titling}
\usepackage{bibentry}
\nobibliography* 

\usepackage{authblk}

\newcommand{\Cech}{\v{C}ech}

\newcommand{\eps}{\epsilon}
\newcommand{\N}{\mathbb{N}}
\newcommand{\Z}{\mathbb{Z}}
\newcommand{\R}{\mathbb{R}}

\newcommand{\T}{\mathbb{T}}

\newcommand{\defeq}{\vcentcolon=}

\newcommand\restr[2]{{ \left.\kern-\nulldelimiterspace #1 \right|_{#2}}}

\usepackage[normalem]{ulem}

\newcommand{\families}{\text{Fam}}
\newcommand{\family}{\mathcal{F}}
\newcommand{\familytwo}{\mathcal{G}}
\newcommand{\familythree}{\mathcal{H}}

\newcommand{\syndetic}{\mathcal{S}}
\newcommand{\thick}{\mathcal{T}}
\newcommand{\PS}{\mathcal{PS}}

\newcommand{\dcsyndetic}{dc\mathcal{S}}
\newcommand{\dcsyndeticbohr}{\dcsyndetic_{\text{Bohr}_0}}
\newcommand{\dcthick}{dc\mathcal{T}}
\newcommand{\dcthickbohr}{\dcthick_{\text{Bohr}_0}}

\newcommand{\central}{\mathcal{C}}

\newcommand{\IP}{\mathcal{IP}}
\newcommand{\SG}{\text{SG}}

\newcommand{\bohrzero}{Bohr$_0$}

\newcommand{\classcap}{\curlyvee}
\newcommand{\classcapdual}{\curlywedge}

\newcommand{\RP}{\mathbf{RP}}
\newcommand{\seq}{S_\text{eq}}
\newcommand{\powerset}{\mathcal{P}}
\newcommand{\Id}{\text{Id}}

\newcommand{\lean}[1]{}
\newcommand{\leanok}{}

\begin{document}

\title{The local dynamical structure of $\Delta^*$ sets via a new Furstenberg family algebra}

\author[1]{Angelina Blahodatna}
\author[1]{Lauren Detmold}
\author[1]{Daniel Glasscock}
\author[2]{Anh N. Le}

\affil[1]{\small Dept. of Mathematics and Statistics, U. of Massachusetts Lowell, Lowell, MA, USA}
\affil[2]{\small Dept. of Mathematics,
	U. of Denver, Denver, CO, USA}

\date{}

\maketitle




\begin{abstract}
    In this paper, we strengthen the connection between the combinatorics of difference sets and the dynamics of group rotations.
    Our main result shows that sets which have non-empty intersection with all difference subsets of a commutative semigroup possess local Bohr structure.
    This generalizes results of Bergelson, Furstenberg, and Weiss and Host and Kra from the integers to arbitrary commutative semigroups.
    We accomplish this by A) utilizing a recent result showing that the regionally proximal relation is an equivalence relation for minimal actions of commutative semigroups and by B) describing a new, DeMorgan-type algebra on Furstenberg families that allows for efficient manipulation and computation. 
    
    We formally verify all of the results in this paper in Lean.  The main results are verified in a Palomar submission, and we link to a Github repository containing code for the complete verification.
\end{abstract}

\clearpage

\tableofcontents

\section{Introduction}

The results in this paper concern patterns in semigroups that are well-known in the integers but less so in general.  For clarity, we quickly dispense with the necessary definitions; motivation for them will follow later in the introduction.

Let $(S,\cdot)$ be a (discrete) semigroup.  A set $A \subseteq S$ is \emph{syndetic} if there exists a finite set $F \subseteq S$ such that $\cup_{f \in F} f^{-1} A = S$; a \emph{$\Delta$ set} if there exists a sequence $s_1, s_2, \ldots \in S$ such that $s_j \in s_i A$ for all $1 \leq i < j$; and a \emph{\bohrzero{}} set if there exists $d \in \N$, a semigroup homomorphism $\varphi: S \to \R^d / \Z^d$, and an open set $U$ containing $\vec 0$ in $\R^d / \Z^d$ such that $\varphi^{-1} U \subseteq A$.  Dually, a subset of $S$ is a \emph{thick set} / \emph{$\Delta^*$ set} / \emph{set of Bohr recurrence} if it has non-empty intersection with all syndetic / $\Delta$ / \bohrzero{} subsets of $S$, respectively.

\subsection{Bohr and difference sets in the integers}

Let $\vec v \in \R^d / \Z^d$ be a point in the $d$-dimensional torus, and let $\eps > 0$.
The set of integers $n$ for which $n \vec v$ is within $\eps$ of $\vec 0$ is an example of a \bohrzero{} subset of $\Z$.
Describing patterns in these \emph{almost periodic} sets has long been important to number theorists and dynamicists.
\bohrzero{} sets -- together with their higher order analogues, \emph{nil-Bohr sets} -- are the essential ordered objects in the theory that describes the structure of arbitrary dynamical systems.

One such pattern that is arises frequently in Ramsey theory and dynamics is described by $\Delta$ sets.  In the integers, these are \emph{infinite difference sets}: sets of the form $(A - A) \cap \N$ where $A \subseteq \Z$ is infinite.
The eponymous Ramsey's Theorem is equivalently formulated to say that one piece of any finite partition of an infinite difference set contains an infinite difference set.
In dynamics, difference sets capture times of single recurrence: note that $n \in A - A$ if and only if $A \cap (A - n) \neq \emptyset$, marking a time $n$ at which the set $A$ ``recurs'' by shifting to intersect itself non-trivially.
Sets which have non-empty intersection with all $\Delta$ sets, that is, $\Delta^*$ sets, are combinatorially large and rich: they are syndetic and contain many additive patterns.

It is easy to see that a \bohrzero{} set is a $\Delta^*$ set. Indeed, consider the \bohrzero{} set $B$ described in the first paragraph above, and fix a sequence of integers $n_1, n_2, \ldots$. We must show that some $n_j - n_i$ belongs to $B$.  The points $n_1 \vec v$, $n_2 \vec v$, $n_3 \vec v$, \dots{} live in the totally bounded space $\R^d / \Z^d$, so some $n_i \vec v$ and $n_j \vec v$ are within $\eps$ of each other.  Subtracting, the point $(n_j - n_i) \vec v$ is within $\eps$ of $\vec 0$, and hence $n_j - n_i \in B$, as was to be shown.
There are other well-known connections between \bohrzero{} sets and $\Delta^*$ sets in the integers, most notably the characterization of minimal equicontinuous systems as precisely those systems in which every point returns to every neighborhood of itself along a $\Delta^*$ set; see \cref{sec_dynamics} for the terminology and \cite[Section 0]{bergelson_leibman_2018} for more details.

Conversely, not every $\Delta^*$ set is a \bohrzero{} set: the set $\Z \setminus \{n^2 \mid n \in \N \}$ is $\Delta^*$ (the set of squares do not contain any difference infinitely often) but is not a \bohrzero{} set (the classic polynomial equidistribution theorem of Weyl implies that \bohrzero{} sets contain perfect squares).
Nevertheless, it is the case that $\Delta^*$ have ``local'' \bohrzero{} structure.  Bergelson, Furstenberg, and Weiss \cite{bergelson_furstenberg_weiss_2006} made this precise for the related family of $\Delta_0^*$ sets (see \cref{sec_diff_sets_and_bz_sets} for the definition); Host and Kra \cite{host_kra_2011} improved their result to address $\Delta^*$ sets and their higher-order analogues.  The following theorem is a consequence of their main result.

\begin{theorem}[{cf. \cite{host_kra_2011}}]
\label{thm_bfw_pw_bohr}
    Let $A \subseteq \Z$.  If $A$ is a $\Delta^*$ set, then for all thick sets $H \subseteq \Z$, there exists a thick set $H' \subseteq H$ and a \bohrzero{} set $B \subseteq \Z$ such that $A \cap H' = B \cap H'$.
\end{theorem}

To put a simpler version of the result loosely: \emph{every $\Delta^*$ set is equal to a \bohrzero{} set on some thick set}.
For our example above, the set $\Z \setminus \{n^2 \mid n \in \N \}$ is equal to $\Z$, a \bohrzero{} set, on itself, a thick set.  If passing to a thick set is akin to looking locally, the theorem says more: \emph{every localization of every $\Delta^*$ set is locally \bohrzero{}}.

\subsection{Local dynamical structure of $\Delta^*$ sets in semigroups}

Our main result is a generalization of \cref{thm_bfw_pw_bohr} to commutative semigroups.

\begin{Maintheorem}
\label{mainthm_gen_of_bfw}
    Let $S$ be a commutative semigroup.  If $A \subseteq S$ is a $\Delta^*$ set, then for all thick sets $H \subseteq S$, there exist a thick set $H' \subseteq H$ and a \bohrzero{} set $B \subseteq S$ such that $A \cap H' = B \cap H'$.
\end{Maintheorem}

To extend \cref{thm_bfw_pw_bohr} to commutative semigroups, we take an entirely different approach to the ergodic-theoretic ones taken by Bergelson, Furstenberg, and Weiss \cite{bergelson_furstenberg_weiss_2006} and Host and Kra \cite{host_kra_2011}.
Our proof relies on the combinatorics of the regionally proximal relation (cf. \cref{thm_delta_set_builder_prestep} and \cref{thm_delta_from_kronecker}), eschewing ergodic theory in favor of a purely topological-dynamical path.
The first step is to reformulate \cref{mainthm_gen_of_bfw}.  Instead of working with $\Delta^*$ and \bohrzero{} sets, we found it best to work with their duals: $\Delta$ sets and sets of Bohr recurrence.

\begin{Maintheorem}
\label{mainthm_gen_of_bfw_reformulated}
    Let $S$ be a commutative semigroup, and let $A \subseteq S$.  If there exists a thick set $H \subseteq S$ such that for all thick sets $H' \subseteq H$, the set $A \cap H'$ is a set of Bohr recurrence, then the set $A$ is a $\Delta$ set.
\end{Maintheorem}

The intrepid reader will pause to show that Theorems \ref{mainthm_gen_of_bfw} and \ref{mainthm_gen_of_bfw_reformulated} are equivalent by means of elementary combinatorial manipulations and the definitions.
All others can rest assured: one of the main contributions of this paper is a description of a new algebra on Furstenberg families that makes this equivalence, and others like it in Ramsey Theory, readily apparent.  We describe that algebra below in \cref{sec_intro_a_new_algebra}.

To prove \cref{mainthm_gen_of_bfw_reformulated}, we start by showing that it suffices to consider sets $A$ of dynamical origin, that is, sets of visit times of a point to a clopen set in a minimal topological dynamical system (see \cref{thm_lspwf_upgrade}).
We then rely crucially on the main result of \cite{blahodatna_detmold_glasscock_le_RP_2026} -- a companion paper in which we show that the regionally proximal relation for minimal actions of commutative semigroups is an equivalence relation -- to find $\Delta$ sets in those sets of dynamical origin (see \cref{thm_delta_set_builder_prestep,thm_delta_from_kronecker}).
Precedent comes from work of Huang, Shao, and Ye \cite[Theorem E]{huang_shao_ye_2016}, who connect the combinatorics of visit time sets with the regionally proximal relation and its higher-order relatives.

Our method is flexible and can be used to demonstrate other kinds of local structure.  
We show, for example, that in countable, commutative semigroups, sets which have non-empty intersection with all central sets, that is, $\central^*$ sets, have local $\IP^*$ structure.  The precise definitions of these families can be found in \cref{sec_char_of_central_star,sec_strong_ip_is_strong_central}; what is important to know for the following theorem is that in most semigroups, not every $\central^*$ set is an $\IP^*$ set. The countability assumption comes from \cref{lemma_thick_conclusion}; we do not know if it is necessary.

\begin{Maintheorem}
\label{mainthm_local_ipstar_in_cstar}
    Let $S$ be a countable, commutative semigroup.  If $A \subseteq S$ is a $\central^*$ set, then for all thick sets $H \subseteq S$, there exists an $\IP^*$ set $B \subseteq S$ and a thick set $H' \subseteq H$ such that $A \cap H' = B \cap H'$.
\end{Maintheorem}

We prove \cref{mainthm_local_ipstar_in_cstar} in \cref{sec_strong_ip_is_strong_central}. To further demonstrate the flexibility of our approach, we outline in \cref{sec_higher_order_analogues} a likely path to recover the full results of Host and Kra in \cite{host_kra_2011} in the integers using higher-order combinatorial input from Huang, Shao, and Ye from \cite{huang_shao_ye_2016}.

\subsection{Furstenberg families as algebraic objects}
\label{sec_ff_as_algebra_objects}

A \emph{Furstenberg family} is an upward-closed collection of subsets of a (discrete) semigroup $S$.
They are named after Hillel Furstenberg, who introduced several important families \cite{Furstenberg81} and used them to describe recurrence properties of topological dynamical systems to great effect.
The collections introduced above -- syndetic sets, thick sets, $\Delta$ sets, $\Delta^*$ sets, \bohrzero{} sets, and sets of Bohr recurrence -- are all examples of Furstenberg families.
The number of such families relevant to combinatorics and dynamics continues to grow as notions are refined and generalized; the curious reader might peruse a recent survey by Hindman \cite{hindman_2020}, where over 50 such families are described.

Because Furstenberg families are upward-closed, membership in a Furstenberg family indicates ``largeness.''
This idea is refined by the notion of filters and partition regular families, both of which are essential to the subject.
A Furstenberg family $\family$ is a \emph{filter} if it is non-empty and for all $A, B \in \family$, we have that $A \cap B \in \family$.
A Furstenberg family $\family$ is \emph{partition regular} if it is not $\powerset(S)$, the power set of $S$, and for all $A, B \subseteq S$ for which $A \cup B \in \family$, we have that $A \in \family$ or $B \in \family$.
(We discuss our particular conventions at greater length in \cref{sec_filters_and_pr_fams_algebra}.)
Partition regular filters -- called \emph{ultrafilters} -- are objects that play an important role in this paper and in topological dynamics more broadly.

While Furstenberg families are organizational and descriptive tools, they are also useful as computational objects.
We will support this assertion by outlining three examples in this subsection and describing a new algebra on families in the next.  The most detailed accounting of an algebra of families that we know of is given by Akin \cite[Ch. 2]{akin_1997}, a source which we cite more specifically numerous times below.

Perhaps the most familiar operation on families, going back at least to Furstenberg \cite[Ch. 9.1]{furstenberg_book_1981}, is the \emph{family dual} map, $*: \powerset(\powerset(S)) \to \powerset(\powerset(S))$, defined by 
\begin{align}
\label{eqn_fam_dual_def}
    \family^* = \big\{A \subseteq S : \forall B \in \family, \ A \cap B \neq \emptyset\big\}.
\end{align}
It is not hard to check that the dual is an antitone\footnote{For all Furstenberg families $\family$ and $\familytwo$, $\family \subseteq \familytwo$ implies $\familytwo^* \subseteq \family^*$.} involution on the subset of $\powerset(\powerset(S))$ of Furstenberg families.
Restricted to filters, the dual is a bijection between filters and partition regular families, and ultrafilters are precisely the self-dual filters.
The dual of $\family$ is often denoted by $k\family$ in the literature, as it is by Akin, who goes on to describe a generalized dual ``with respect to a family''; see \cite[Pg. 27]{akin_1997}.

Another frequent operation on Furstenberg families is best introduced by way of the notion of piecewise syndeticity.  
A set is \emph{piecewise syndetic} if it is equal to the intersection of a syndetic set and a thick set.
It is a fact that the family of piecewise syndetic sets is partition regular; in the integers, this result can be traced back to Brown \cite[Lemma 1]{brown_1971}, and it has appeared since then numerous times independently in various, more general forms.
Inspired by this, for an arbitrary Furstenberg family $\family$, we are led to consider the family comprised of intersections of members of $\family$ with members of $\family^*$.
This family is, incredibly, always partition regular; see \cite[Prop. 2.1 (e)]{akin_1997}, \cite[Prop. 2.5 (h)]{christopherson_johnson_2022}.
Given families $\family$ and $\familytwo$, the more general operation is that of forming the family of sets of the form $A \cap B$ where $A \in \family$ and $B \in \familytwo$.
Akin denotes this by operation by $\family \cdot \familytwo$; we discuss it in the next subsection.

A third example of Furstenberg families as computational objects is derived from the algebra of ultrafilters.
The binary operation on a semigroup $S$ is well-known to lift to a binary operation on $\beta S$, the collection of ultrafilters on $S$.
The resulting algebra has been intensively studied and has found myriad applications in combinatorics and dynamics \cite{ultrafilters_book_2010,hindman_strauss_book_2012}.
The same operation defining the product of ultrafilters meaningfully describes the product of arbitrary Furstenberg families.
The third and fourth authors explored some of this algebra and its applications in \cite{glasscock_le_2025,glasscock_le_synd_idem_filters_2025}.
Though we do not develop this direction in this paper, it would be interesting to understand how this algebra of families interacts with the algebra described next (see \cref{sec_family_meet_with_algebra}).

\subsection{A new Furstenberg family algebra via family meet and join}
\label{sec_intro_a_new_algebra}

We introduce an algebra on Furstenberg families based on two binary operations which we call \emph{family meet and join}.
These operations are best defined more broadly on $\powerset(\powerset(S))$, then restricted to $\families(S)$, the collection of Furstenberg families of $S$.  While the historical context for Furstenberg families is for collections of subsets of discrete semigroups, the abstract results we deal with here do not use the semigroup structure on $S$.  We proceed in this section by assuming only that $S$ is a set.

For $\family, \familytwo \in \powerset(\powerset(S))$, we define the \emph{family meet} operator, $\classcapdual: \mathcal{P}(\mathcal{P}(S))^2 \to \mathcal{P}(\mathcal{P}(S))$, by
   \[\family \classcapdual \familytwo = \big\{A \subseteq S : \forall B \in \family^*, \ A \cap B \in \familytwo \big\}\]
and the \emph{family join} operator, $\classcap: \powerset(\powerset(S))^2 \to \powerset(\powerset(S))$, by
    \[\family \classcap \familytwo = \big\{C \subseteq S : \exists A \in \family, \ \exists B \in \familytwo, \ C = A \cap B\big\}.\]
The following theorem, proved in \cref{sec_algebra_structure_of_ff}, summarizes the main algebraic properties of these operations.
Recall the definition of the family dual ($\ast$) defined in \eqref{eqn_fam_dual_def}.

\begin{Maintheorem}
    \label{mainthm_algebra_properties}
    Let $S$ be a non-empty set, and denote by $\families(S) \subseteq \powerset(\powerset(S))$ the collection of Furstenberg families of $S$. Family meet ($\classcapdual$) and family join ($\classcap$), restricted to $\families(S)^2$, are associative, commutative, and monotone operators that, together with the family dual ($\ast$), satisfy the De Morgan-type laws
    \begin{align}
    \label{eqn_de_morgan_type_laws}
    (\family \classcap \familytwo)^* = \family^* \classcapdual \familytwo^* \qquad \text{ and } \qquad (\family \classcapdual \familytwo)^* = \family^* \classcap \familytwo^*.
    \end{align}
\end{Maintheorem}

While the proof of \cref{mainthm_algebra_properties} is entirely elementary, we have found the resulting algebra to be extraordinarily helpful in formulating and re-formulating theorems and conjectures concerning the relationships between Furstenberg families.
A illustrative example: to demonstrate that the family $\family \classcap \family^*$ is partition regular -- and thus, for example, to conclude that the family of piecewise syndetic sets is partition regular -- takes little more than demonstrating that $\family \classcapdual \family^*$ is a filter, a one-line exercise from the definitions.
We elaborate on this in \cref{sec_recitation}, where we work by way of examples to convince the skeptical reader as to the utility of the algebra.

As discussed in the previous subsection, the family join operator has appeared several times before in the literature; Akin and Glasner \cite{glasner_akin_2016} call it \emph{join}.
As far as we know, the family meet operator, as written here, is new, though it could be described equivalently in terms of Akin's generalized family dual \cite[Pg. 27]{akin_1997}.
We note that our algebra differs in emphasis from Akin's in that it places filters and partition regular families on equal footing by focusing on family meet as a binary operator.  This change in emphasis aids in a different variety of applications.

Despite what may be suggested by the notation and terminology, it is not the case that $(\families(S),\classcapdual,\classcap)$ forms a lattice; see \cref{sec_filters_and_pr_fams_algebra}.  
Nevertheless, beyond the basic properties described in \cref{mainthm_algebra_properties}, there are a number of true sentences in the algebra involving arbitrary families and family dual, meet, and join.  For example, for all families $\family$, $\familytwo$, and $\familythree$,
\begin{enumerate}
    \item $\family \classcap (\familytwo \classcapdual \familythree) \subseteq (\family \classcap \familytwo) \classcapdual \familythree$; 
    
    \item $\family^* \classcap (\family \classcapdual \familytwo) \subseteq \familytwo$; and
    
    \item
    $\familythree^* \subseteq \family \classcapdual (\family^* \classcap \familytwo^*) \text{ if and only if } \family \classcapdual \familytwo \subseteq \family \classcapdual \familythree$.
\end{enumerate}

We do not attempt a systematic study of all such true sentences and leave the elementary proofs of these to the curious reader.
Instead, in this paper, we focus on the algebra that leads us to examples and applications.
For example, the DeMorgan-type laws apply to unions and intersections of families as expected (\cref{lemma_demorgan_for_union_and_int}):
\[\bigg( \bigcap_i \family_i \bigg)^* = \bigcup_i \family_i^* \qquad \text{ and } \qquad \bigg( \bigcup_i \family_i \bigg)^* = \bigcap_i \family_i^*.\]
Combined with the fact that the family dual is antitone and the operator DeMorgan-type laws from \cref{mainthm_algebra_properties}, we can quickly show that Theorems \ref{mainthm_gen_of_bfw} and \ref{mainthm_gen_of_bfw_reformulated} are equivalent.
In the proof that follows, we denote the families of syndetic sets, thick sets, $\Delta$ sets, $\Delta^*$ sets, \bohrzero{} sets, and sets of Bohr recurrence by $\syndetic$, $\thick$, $\Delta$, $\Delta^*$, $\dcsyndeticbohr$, and $\dcthickbohr$, respectively.

\begin{proof}[Proof that Theorems \ref{mainthm_gen_of_bfw} and \ref{mainthm_gen_of_bfw_reformulated} are equivalent]
    A short exercise from the definitions (see \cref{lemma_desc_of_very_strong}) shows that \cref{mainthm_gen_of_bfw} stated in terms of family meet and join is
    \[\Delta^* \subseteq \bigcap_{H \in \thick} \big( \ \thick_{\cap H} \classcap \dcsyndeticbohr \ \big),\]
    where $\thick_{\cap H}$ denotes the family of sets $A \subseteq S$ for which $A \cap H \in \thick$.
    Applying the dual, the antitone property, and the DeMorgan-type laws, we have that this is equivalent to
    \[\bigcup_{H \in \thick} \big( \ (\thick_{\cap H})^* \classcapdual \dcthickbohr \ \big) \subseteq \Delta,\]
    which is precisely \cref{mainthm_gen_of_bfw_reformulated}.
\end{proof}

\subsection{Outline}

The paper is organized as follows.  Preliminaries on semigroups, topological dynamics, and ultrafilters are given in \cref{sec_prelims}.
We describe the algebra of the family dual, meet, and join and prove \cref{mainthm_algebra_properties} in \cref{sec_algebra_structure_of_ff}, then we follow with concrete examples of the utility of the algebra in \cref{sec_recitation}, including several new characterizations of the family of $\central^*$ sets.
We prove Theorems \ref{mainthm_gen_of_bfw}, \ref{mainthm_gen_of_bfw_reformulated}, and \ref{mainthm_local_ipstar_in_cstar} in \cref{sec_local_dy_struct} and conclude by outlining a number of open directions in \cref{sec_open_directions}.

\subsection{AI use disclosure}

The authors used generative AI (including OpenAI's ChatGPT, Google's Gemini, and Anthropic's Claude) to assist with literature searches and language editing.
All references suggested by these tools were independently verified by the authors.
Generative AI was not used to generate mathematical results or proofs; all mathematical content is the work of the authors, and the authors take full responsibility for it.

\subsection{Lean verification}

We verified all of the results in this paper in Lean.  The verification code can be found on Github at: \url{https://github.com/dglasscockUML/NSF_LEAPS}.  In addition, we registered the main results appearing in the introduction with Palomar (\url{https://palomar-registry.org/}).  A complete link to the Palomar registry entry will be provided in an updated version of this paper.

\subsection{Acknowledgments}

The first three authors gratefully acknowledge support from the US National Science Foundation under a LEAPS Grant, Award Number 2418589.

\section{Preliminaries}
\label{sec_prelims}

\subsection{Sets and semigroups}

Let $S$ and $X$ be sets.  When $S$ acts by maps with domain $X$, we write $sx$ for the image of $x \in X$ under the map $s \in S$.  For $H \subseteq S$ and $A \subseteq X$, we write
\begin{center}
\begin{tabular}{r@{\;}c@{\;}l@{\qquad}r@{\;}c@{\;}l}
$Hx$ & $\defeq$ & $\{sx \mid s \in H\},$
&
$sA$ & $\defeq$ & $\{sx \mid x \in A\},$
\\
$HA$ & $\defeq$ & $\{sx \mid s \in H,\ x \in A\},$
&
$s^{-1}A$ & $\defeq$ & $\{x \in X \mid sx \in A\}.$
\end{tabular}
\end{center}
The set $A$ is \emph{$S$-invariant} if $SA \subseteq A$. The set $S$ will always act on $X^2$ via the diagonal action $s(x,y) = (sx,sy)$.  We refer throughout to the set $\{(x,x) : x \in X\}$ as \emph{the diagonal} and denote it by $\Delta_X$ or simply $\Delta$.

The semigroups in this paper are all non-empty and discrete, written multiplicatively, and are not assumed to contain an identity element.

\subsection{Topological dynamics}
\label{sec_dynamics}

Throughout the paper, we work with compact Hausdorff spaces. Although we do not use the language of uniform spaces, we note that compact Hausdorff spaces are precisely compact uniform spaces. Indeed, every compact Hausdorff space admits a unique uniformity given by the open sets containing the diagonal, whereas every compact uniform space is compact and completely regular, and therefore Hausdorff.

\begin{lemma}
\label{lemma_uniform_near_diag}
\lean{nbhdOfDiagForcesOtherSetContainment}
\leanok{}
Let $X$ be a compact Hausdorff space.  Let $y \in X$ and $U \subseteq X$ be open and contain $y$.  There exists an open set $V \subseteq U$ containing $y$ and an open set $\alpha \subseteq X^2$ that contains the diagonal such that for all $x, z \in X$, if $(x,z) \in \alpha$ and $z \in V$, then $x \in U$.
\end{lemma}

\begin{proof}
\leanok{}
    By the regularity of $X$, there exists an open set $V \subseteq X$ containing $y$ and such that $\overline{V} \subseteq U$.  Put $K = X \setminus U$, and define
    \[\alpha = X^2 \setminus (K \times \overline{V}).\]
    The set $\alpha$ is clearly open, and it contains the diagonal since $K \cap \overline{V} = \emptyset$.  Finally, suppose $x, z \in X$ are such that $(x,z) \in \alpha$ and $z \in V$.  It must be that $x \in X \setminus K = U$, as desired.
\end{proof}

\subsubsection{Systems}
\label{sec_hom_systems}

A \emph{(topological dynamical) system} is an action of a semigroup $S$ on a compact Hausdorff space $X$ by continuous maps.
We say that $X$ is an \emph{$S$-system} and we write the action of $S$ by juxtaposition: for $s \in S$ and $x \in X$, we denote by $sx$ the image of $x$ under the map $s : X \to X$.  The \emph{orbit} and \emph{orbit closure} of $x \in X$ are the sets $Sx$ and $\overline{Sx}$, respectively.

An $S$-system $Y$ is said to be a \emph{factor} of $X$ if there is a continuous surjection $\pi: X \to Y$ satisfying $s (\pi x) = \pi (s x)$ for all $s \in S$ and $x \in X$.
In the event that such a factor map exists that is a homeomorphism, we say that the systems $X$ and $Y$ are \emph{isomorphic}.  A \emph{subsystem} of $X$ is a non-empty, closed, $S$-invariant set $Z \subseteq X$ considered as an $S$-system by restricting the maps $s \in S$ to $Z$.
Orbit closures of points are good examples of subsystems.

Here is a situation we will encounter a few times in the paper.
Let $\varphi: S \to T$ be a semigroup homomorphism.
If $X$ is a $T$-system, then there is a natural way to make $X$ into an $S$-system: for $s \in S$ and $x \in X$, define $sx = \varphi(s)x$.
Since $\varphi$ is a homomorphism, it is immediate that this yields an action of $S$ on $X$ via continuous maps.

\subsubsection{Set and point visit times}

Let $S$ be a semigroup, $X$ be an $S$-system, $x \in X$, and $U, V \subseteq X$.  The \emph{set of visit times} of $x$ to $U$ and $V$ to $U$ are
\begin{align*}
    R(x,U) &\defeq \big\{s \in S \mid sx \in U\big\} \text{ and} \\
    R(V,U) &\defeq \big\{s \in S \mid sV \cap U \neq \emptyset\big\}.
\end{align*}

\begin{lemma}
\label{lemma_return_time_algebra_one}
\lean{visitsToPreimages}
\leanok{}
    Let $S$ be a semigroup, $X$ be a $S$-system, $x \in X$, $U \subseteq X$, and $s \in S$. We have
    \[s^{-1}R(x,U) = R(x,s^{-1}U).\]
\end{lemma}

\begin{proof}
\leanok{}
    This is an easy consequence of the definitions and is left to the reader.
\end{proof}

\begin{lemma}
    \label{lemma_return_to_intersection}
    \label{lemma_return_to_union}
    \lean{visitToInter}
    \lean{visitToUnion}
    \leanok{}
    Let $S$ be a semigroup and $X$ be an $S$-system.  For all $x \in X$ and $U_1, \ldots, U_k \subseteq X$,
    \begin{align*}
        R\left(x, \bigcup_{i=1}^k U_i \right) = \bigcup_{i=1}^k R(x,U_i) \qquad \text{ and } \qquad
        R\left(x, \bigcap_{i=1}^k U_i \right) = \bigcap_{i=1}^k R(x,U_i).
    \end{align*}
\end{lemma}

\begin{proof}
\leanok{}
    This is an easy consequence of the definitions and is left to the reader.
\end{proof}

For the following lemma, recall the definition of a partition regular family from \cref{sec_ff_as_algebra_objects}.

\begin{lemma}
\label{lemma_pr_family_concentration}
\lean{visitTimeConcentrationForPRFamily}
\leanok{}
    Let $X$ be an $S$-system, $x \in X$, and $U \subseteq X$ be non-empty and closed.  Let $\family$ be a partition regular family of subsets of $S$. If $R(x,U) \in \family$, then there exists $y \in U$ such that for all open $V \subseteq X$ containing $y$, $R(x,V) \in \family$.
\end{lemma}

\begin{proof}
\leanok{}
    We prove the contrapositive.  Suppose for all $y \in U$, there exists an open $V_y \subseteq X$ containing $y$ such that $R(x,V_y) \not\in \family$.  By the compactness of $U$, there exists a finite set $F \subseteq X$ such that $U \subseteq \bigcup_{y \in F} V_y$. By \cref{lemma_return_to_union}, we see that
    \[
        R(x,U) \subseteq R \left(x,\bigcup_{y \in F} V_y \right) = \bigcup_{y \in F} R(x,V_y).
    \]
    Since no $R(x,V_y)$ is in $\family$ and $\family$ is partition regular, we have that $\bigcup_{y \in F} R(x,V_y) \not\in \family$, and hence that $R(x,U) \not\in \family$, as desired.
\end{proof}

\subsubsection{Uniform recurrence and minimality}

An $S$-system $X$ is \emph{minimal} if the orbit $Sx$ of every point $x \in X$ is dense in $X$. The following facts follow from the standard arguments:
\begin{enumerate}
    \item \label{item_min_iff_triv_subsystems}
    a system $X$ is minimal if and only if $X$ is the only subsystem of $X$, and
    
    \item \label{item_every_system_has_minimal}
    every system $X$ has a minimal subsystem.
\end{enumerate}
If $\pi: X \to Y$ is a factor map of $S$-systems and $Z$ is a subsystem of $X$, then $\pi Z$ is a subsystem of $Y$.  It follows from this and Fact \ref{item_min_iff_triv_subsystems}. that if $Z$ is a minimal subsystem of $X$, then $\pi Z$ is a minimal subsystem of $Y$.

If for all open $U \subseteq X$ containing a point $x \in X$, the set $R(x,U)$ is syndetic (recall the definition from the introduction), the point $x$ is called \emph{$S$-uniformly recurrent}. The following facts follow from the standard arguments:
\begin{enumerate}
    \setcounter{enumi}{2}
    \item \label{item_ur_implies_min}
    if $x$ is uniformly recurrent, then $\overline{Sx}$ is a minimal $S$-system, and
    
    \item \label{item_minimal_implies_ur}
    all points in a minimal system are uniformly recurrent.
\end{enumerate}
All four facts above will be used in this paper without special mention.

\begin{lemma}
\label{lemma_minimal_comm_action_implies_surjective}
\lean{minimalCommActionIsSurjective}
    Let $S$ be a commutative semigroup and $X$ be a minimal $S$-system.  For all $s \in S$, the map $s: X \to X$ is surjective.
\end{lemma}

\begin{proof}
\leanok{}
    Let $s \in S$.  Since $S$ is commutative, the set $sX$ is a closed, $S$-invariant subset of $X$.  By minimality, $sX = X$, showing that $s$ is surjective, as desired.
\end{proof}

\begin{theorem}
    \label{thm_ur_lift}
    \lean{liftUniformRecurrentPoint}
    \leanok{}
    Let $S$ be a semigroup, and let $\pi: X \to Y$ be a factor map of $S$-systems.  For all $S$-uniformly recurrent points $y \in Y$, there exists an $S$-uniformly recurrent point $x \in X$ such that $\pi x = y$.
\end{theorem}

\begin{proof}
    \leanok{}
    Let $y \in Y$ be $S$-uniformly recurrent. Let $Z$ be the $S$-orbit closure of $y$.  It is easy to check that $\pi^{-1} Z$ is a subsystem of $X$.  By Fact \ref{item_every_system_has_minimal}. above, there exists a minimal subsystem $W \subseteq \pi^{-1} Z$ of $\pi^{-1} Z$ and, hence, of $X$.  Since images under factor maps of minimal systems are minimal systems, we have that $\pi W \subseteq Z$ is a minimal subsystem of $Y$.  But $Z$ is minimal, so $\pi W = Z$.  Since $y \in Z$, there exists $x\in W$ such that $\pi x = y$.  Since $W$ is minimal and $x \in W$, we see by Fact \ref{item_minimal_implies_ur}. that $x$ is $S$-uniformly recurrent, as desired.
\end{proof}

\subsection{Equicontinuity and Bohr sets}

The results in this section are standard for group actions but less well known or entirely new for semigroup actions.  Throughout, as above, $S$ is a semigroup and $X$ is an $S$-system.

\subsubsection{Equicontinuity, proximality, and distality}

An $S$-system $X$ is \emph{equicontinuous} if the family $\{s: X \to X \mid s \in S\}$ is (uniformly) equicontinuous: for all open $\beta \subseteq X^2$ containing the diagonal, there exists an open $\alpha \subseteq X^2$ containing the diagonal such that $S \alpha \subseteq \beta$.  It is quick to check that subsystems, factors, and arbitrary products of equicontinuous systems are equicontinuous.

\begin{lemma}
    \label{lemma_hom_system_facts}
    \lean{homSystemOfEquicontinuousSystemIsEquicontinuous}
    \leanok{}
    Let $\varphi: S \to T$ be a semigroup homomorphism.  If $X$ is an equicontinuous $T$-system, then via the homomorphism $\varphi$, it is an equicontinuous $S$-system.
\end{lemma}

\begin{proof}
\leanok{}
    By definition, $X$ is an equicontinuous $S$-system if $\{s : X \to X \mid s \in S\}$ is a (uniformly) equicontinuous family of maps.  By the definition of the action, this family of maps is a subset of $\{t : X \to X \mid t \in T\}$, which is (uniformly) equicontinuous by assumption.
\end{proof}

Points $x$ and $y$ in an $S$-system $X$ are \emph{proximal} if for all open $\alpha \subseteq X^2$ containing the diagonal, there exists $s \in S$ such that $s(x,y) \in \alpha$, where we recall that $S$ acts on $X^2$ via the diagonal action.
It is well-known that if the points $x$ and $y$ are proximal, then for all open $\alpha \subseteq X^2$ containing the diagonal, the set $R\big((x,y), \alpha \big)$ is a thick subset of $S$.

A system is called \emph{distal} if no two distinct points are proximal.
That distal systems are semisimple is a well-known fact for group actions; the following lemma records this fact for semigroup actions.
We make use a minimal idempotent ultrafilter on $S$; see \cref{sec_ultrafilters} below for the details.

\begin{lemma}
\label{lemma_distal_implies_semisimple}
\lean{distalImpliesSemisimple}
\leanok{}
    Let $S$ be a semigroup and $X$ be an $S$-system.  If $X$ is distal, then every point in $X$ is $S$-uniformly recurrent.
\end{lemma}

\begin{proof}
\leanok{}
    Suppose $X$ is distal, and let $x \in X$.  Let $p \in \beta S$ be any minimal, idempotent ultrafilter. The point $px$ is uniformly recurrent and proximal to $x$.  Since $X$ is distal and $(x,px)$ is a proximal pair, it must be that $px = x$.  Therefore, $x$ is uniformly recurrent, as desired.
\end{proof}

\subsubsection{The regionally proximal and equicontinuous structure relations}

The regionally proximal relation and equicontinuous structure relation are fundamental relations capturing equicontinuous behavior in $X$.  Both are well-understood for group actions but relatively unexplored for semigroup actions; see \cite{auslander_dai_2019,blahodatna_detmold_glasscock_le_RP_2026,dai_xiao_2017}.

Let $X$ be an $S$-system. The point $x$ is \emph{regionally proximal to $y$}, that is, $(x,y) \in \RP$, if for all open $U \subseteq X$ containing $x$, open $V \subseteq X$ containing $y$, and open $\alpha \subseteq X^2$ containing the diagonal, there exist $x' \in U$, $y' \in V$, and $s \in S$ with
\[
    (sx', sy') \in \alpha.
\]
The point $(x,y)$ belongs to the equicontinuous structure relation, that is, $(x,y) \in \seq$, if $\pi x = \pi y$ for all equicontinuous factors $\pi: X \to Y$.
It follows from standard arguments that $\seq$ is a closed, $S$-invariant equivalence relation and that $X / \seq$ is the largest equicontinuous factor of $X$; see, eg., \cite{blahodatna_detmold_glasscock_le_RP_2026}.

We will make use of two results from \cite{blahodatna_detmold_glasscock_le_RP_2026} concerning the relationships between these relations.
The first is a sufficient condition for the relations to coincide.

\begin{theorem}[{\cite[Theorem B]{blahodatna_detmold_glasscock_le_RP_2026}}]
    \label{mainthm_RP_is_eq_relation}
    Let $S$ be a commutative semigroup and $X$ be a minimal $S$-system.  The regionally proximal relation coincides with the equicontinuous structure relation:
    \[\RP = \seq.\]
    In particular, $\RP$ is an equivalence relation and $X \to X/\RP$ is the largest equicontinuous factor of $X$.
\end{theorem}

For our purposes in this paper, the precise definition of $\RP$ is not important; we will need only the combinatorial consequence from the following result.

\begin{theorem}[{\cite[Theorem 5.5]{blahodatna_detmold_glasscock_le_RP_2026}}]
\label{thm_delta_set_builder_prestep}
    Let $S$ be a commutative semigroup, $X$ be a minimal $S$-system, and $x, y \in X$.  If $(x,y) \in \RP$, then for all non-empty, open $V, W \subseteq X$ with $y \in V$, the set
    \[R(x,W) \cap R(V,W)\]
    is syndetic.
\end{theorem}

\subsubsection{Compact subsemigroups of $C(X,X)$}

For $X$ a compact Hausdorff space, we denote by $C(X,X)$ the set of continuous functions $X \to X$ equipped with the topology of uniform convergence.  The main result in this subsection is \cref{thm_compact_abelian_grp_in_cxx}, that the closure of an equicontinuous subsemigroup of $C(X,X)$ that consists of surjections is a topological group.
This result is closely related to the well-known fact that the Ellis enveloping semigroup of a minimal system is a compact topological group if and only if the action is equicontinuous \cite[Thm. 3.34]{devries_book_1993}.
We need \cref{thm_compact_abelian_grp_in_cxx} to show the equivalence between the two natural definitions of \bohrzero{} subsets of a semigroup $S$ in \cref{sec_bohr_def}.

In what follows, a \emph{topological semigroup} is a semigroup together with a topology with respect to which multiplication $(s,t) \mapsto st$ is continuous.  A \emph{topological group} is a group together with a topology with respect to which multiplication $(g,h) \mapsto gh$ and inversion $g \mapsto g^{-1}$ are continuous.

\begin{lemma}
\label{thm_top_group_suff_condition}
\lean{groupFromCompactSubmonoidIsTopologicalGroup}
    Let $S$ be a compact subsemigroup of a Hausdorff topological semigroup $T$ with identity $e_T$.
    If $e_T \in S$ and $e_T$ is the unique idempotent of $S$, then $S$ is a compact Hausdorff topological group with identity $e_T$.
\end{lemma}

\begin{proof}
\leanok{}
    We will show that
    \begin{enumerate}
        \item every element of $S$ has a unique two-sided inverse in $S$, and
        \item the map $S \to S$ sending an element to its two-sided inverse in $S$ is continuous.
    \end{enumerate}
    This suffices to reach the conclusion of the lemma.  Indeed, by assumption, we have that $S$ is compact and Hausdorff. Since $S$ is a semigroup, the product of elements of $S$ is already in $S$.  Since every element has a two-sided inverse in $S$ and $S$ contains an identity $e_T$, it is a group with identity $e_T$.  Since $T$ is a topological semigroup, multiplication restricted to $S \times S$ is continuous.  Since the inverse map on $S$ is continuous, it is a compact Hausdorff topological group.

    To see 1., let $s \in S$.  Consider $A \defeq \overline{\{s^n \mid n \geq 2\}}$.  It is a non-empty, compact subsemigroup of $S$.
    By the Ellis-Numakura lemma (cf. \cite[Thm. 2.5]{hindman_strauss_book_2012}), the semigroup $A$ contains an idempotent.
    By assumption, the identity $e_T$ is the unique idempotent in $S$.
    Since $A \subseteq S$, we have that $e_T \in A$.
    We claim that $A \subseteq sS$.  Indeed, $S$ is compact and left multipliction by $s$ is continuous, so $sS$ is compact, hence closed.  Clearly $\{s^n \mid n \geq 2\} \subseteq sS$, so $A \subseteq sS$.
    Since $e_T \in A$, we have that $e_T \in sS$, meaning there exists $s_1 \in S$ such that $ss_1 = e_T$.

    By the same argument using right multiplication by $s$ instead of left, we get that there exists $s_2 \in S$ such that $s_2s = e_T$. From $ss_1 = s_2s = e_T$, we see $s_1 = s_2 s s_1 = s_2$, whence $s_1 = s_2$ is a two-sided inverse in $S$ for $s$. If $s_3 \in S$ is a two-sided inverse for $s$, then $s_1 = s_3 s s_1 = s_3$, whereby $s_1$ is the unique two-sided inverse for $s$ in $S$.

    Now 2. follows a well-known theorem of Ellis \cite{ellis_1957} that a group with a locally compact topology in which multiplication on the left and right is continuous is a topological group.  Alternatively, as an exercise, the closed graph theorem can be used to show 2. in a few lines.
\end{proof}

\begin{theorem}
\label{thm_compact_abelian_grp_in_cxx}
\lean{isTopologicalGroupFromSurjectiveSubsemiOfCXX}
\leanok{}
    Let $X$ be a compact Hausdorff space.  Let $S \subseteq C(X,X)$ be a commutative subsemigroup consisting of surjections of $X$.  If $\overline{S}$ is compact, then $\overline{S}$ is a compact Hausdorff abelian group with identity $\Id_X$.
\end{theorem}

\begin{proof}
\leanok{}
    Suppose $\overline{S}$ is compact. We will show that
    \begin{enumerate}
        \item $\overline{S}$ is a commutative subsemigroup of $C(X,X)$, and that

        \item if $h \in \overline{S}$ is idempotent, then $h = \Id_X$.
    \end{enumerate}
    This suffices to reach the conclusion of the theorem.  Indeed, we get from 1. that $\overline{S}$ is a compact subsemigroup of the Hausdorff topological semigroup $C(X,X)$, which has identity $\Id_X$.
    By the Ellis-Numakura lemma (cf. \cite[Thm. 2.5]{hindman_strauss_book_2012}), the set $\overline{S}$ contains an idempotent element $h$.
    From 2., we see that $h = \Id_X \in \overline{S}$ and $\Id_X$ is the unique idempotent in $\overline{S}$.
    It follows from \cref{thm_top_group_suff_condition} that $\overline{S}$ is a compact Hausdorff topological group with identity $\Id_X$.
    Since $\overline{S}$ was shown to be commutative, we have that, in fact, $\overline{S}$ is a compact Hausdorff abelian group.

    To see 1., denote the multiplication on $C(X,X)$ (that is, composition of maps) by $M: C(X,X)^2 \to C(X,X)$. Since $M$ is continuous and $M (S \times S) \subseteq S$ (since $S$ is a semigroup), we have that
    \[M (\overline{S} \times \overline{S}) = M \overline{(S \times S)} \subseteq \overline{ M(S \times S) } \subseteq \overline{ S }.\]
    Therefore, we conclude that $\overline{S}$ is a subsemigroup of $C(X,X)$.

    Let $\varphi: S \times S \to S \times S$ be $\varphi: (s,t) \mapsto (t,s)$, and define $\tilde M : C(X,X)^2 \to C(X,X)$ by $\tilde M = M \circ \varphi$.
    Since $\varphi$ and $M$ are continuous, so is $\tilde M$.
    We wish to show that $\overline{S}$ is commutative, that is, $\restr{M}{\overline{S} \times \overline{S}} = \restr{\tilde M}{\overline{S} \times \overline{S}}$.
    Since $S$ is a commutative semigroup, $\restr{M}{S \times S} = \restr{\tilde M}{S \times S}$.
    Since $S \times S$ is a dense subset of $\overline{S} \times \overline{S}$ and $\restr{M}{\overline{S} \times \overline{S}}$ and $\restr{\tilde M}{\overline{S} \times \overline{S}}$ are continuous functions, they are equal on $\overline{S} \times \overline{S}$, as desired.
    
    To see 2., we start by showing that $\overline{S}$ consists of surjections. Denote by $2^X$ the set of non-empty, compact subsets of $X$.
    Given the Vietoris topology, it is a compact Hausdorff space, and the map $\psi: C(X,X) \to 2^X$ given by $f \mapsto fX$ is continuous.  
    By assumption, $S$ consists of surjections, so $\restr{\psi}{S}$ is equal to the constant $X$ function on $S$.  Since $\psi$ and the constant $X$ function are continuous, and since $S$ is a dense subset of $\overline{S}$, we get that $\restr{\psi}{\overline{S}}$ is equal to the constant $X$ function on $\overline{S}$.
    Therefore, all maps in $\overline{S}$ consist of surjections.

    Suppose $h \in \overline{S}$ is idempotent.
    By the previous paragraph, the map $h$ is surjective.  We claim that $h = \Id_X$.
    Indeed, for all $x \in X$, there exists $y \in X$ such that $x = hy = h^2 y = hx$, as was to be shown.
\end{proof}

\subsubsection{A dynamical characterization of \bohrzero{} sets}
\label{sec_bohr_def}

Let $S$ be a (discrete) semigroup, and denote by $\T^d$ the $d$-dimensional torus $\R^d / \Z^d$. As defined in the introduction, a set $A \subseteq S$ is a \emph{\bohrzero{} set} if there exists $d \in \N$, a homomorphism $\varphi: S \to \T^d$, and an open set $U \subseteq \T^d$ containing $\vec 0$ such that $\varphi^{-1} U \subseteq A$.
The next theorem provides a useful equivalent dynamical characterization of \bohrzero{} sets for commutative semigroups, the group-action analogue of which is well-known. It also partially explains the notation $\dcsyndeticbohr$ for the family of \bohrzero{} sets: we define the related family of dynamically central syndetic sets, $\dcsyndetic$, in \cref{sec_char_of_central_star}.

\begin{theorem}
\label{thm_equiv_char_of_bohr_zero_sets}
\lean{bohrZeroSetsContainEquiReturns}
\lean{equiReturnsAreBohrZero}
\leanok{}
    Let $S$ be a (discrete) commutative semigroup, and let $A \subseteq S$.  The following are equivalent:
    \begin{enumerate}
        \item $A$ is \bohrzero{};

        \item \label{item_bohr_equi_action_semi}
        there exists a minimal, equicontinuous $S$-system $X$, a point $x \in X$, and an open set $U \subseteq X$ containing $x$ such that $R(x,U) \subseteq A$.
    \end{enumerate}
\end{theorem}

\begin{proof}
\leanok{}
    1. implies 2. \ By assumption, there exists $d \in \N$, a homomorphism $\varphi: S \to \T^d$, and an open set $U \subseteq \T^d$ containing zero such that $\varphi^{-1} U \subseteq A$.
    As a (discrete) commutative semigroup, $\T^d$ acts via group addition equicontinuously on itself as a compact Hausdorff space.
    By \cref{lemma_hom_system_facts}, the homomorphism $\varphi: S \to \T^d$ makes the compact Hausdorff space $\T^d$ into an equicontinuous $S$-system (cf. \cref{sec_hom_systems}).
    
    Let $x = \vec 0$, $X$ be the $S$-orbit closure of $x$, and $V = X \cap U$.
    We claim that $X$ is a minimal, equicontinuous $S$-system, that $x \in X$, that $V$ is an open subset of $X$ containing $x$, and that $R(x,V) \subseteq A$.

    Since $S$ acts by isometries, the $S$-system $\T^d$ is distal.  By \cref{lemma_distal_implies_semisimple}, the point $x$ is $S$-uniformly recurrent.
    Therefore, $X$ is minimal, $x \in X$, and $V$ is an open subset of $X$ containing $x$.
    That $X$ is equicontinuous follows from the simple fact that subsystems of equicontinuous systems are equicontinuous.

    To see that $R(x,V) \subseteq A$, note first that $R(x,V) \subseteq R(x,U)$, where $R(x,V)$ is considered in the $S$-system $X$ and $R(x,U)$ is in the $S$-system $\T^d$.  
    Then $R(x,U) \subseteq \varphi^{-1} U$ since if $g \in R(x,U)$, then $\varphi(g) = \varphi(g) + \vec 0 = gx \in U$, meaning $g \in \varphi^{-1} U$. Finally, $\varphi^{-1} U \subseteq A$ by assumption.\\

    2. implies 1. \ Suppose there exists a minimal, equicontinuous $S$-system $X$, a point $x \in X$, and an open set $U \subseteq X$ containing $x$ such that $R(x,U) \subseteq A$. Let $i: S \to C(X,X)$ be the semigroup homomorphism sending an element of $S$ to its action map.
    Since $X$ is an equicontinuous $S$-system, the family of maps $iS \subseteq C(X,X)$ is equicontinuous.
    By the Arzela–Ascoli theorem (Dunford-Schwartz 1958, Section IV.6.7), the set $\overline{iS}$ is compact.
    By \cref{lemma_minimal_comm_action_implies_surjective}, since $S$ is commutative and acts minimally, it acts by surjections.  Thus, we see that $iS$ is a commutative subsemigroup of $C(X,X)$ comprised of surjections. By \cref{thm_compact_abelian_grp_in_cxx}, the set $H \defeq \overline{iS}$ is a compact Hausdorff abelian group with identity $\Id_X$.
    
    The map $\xi: H \to X$ defined by $\xi f = fx$ is continuous.  Define
    \[W = \xi^{-1} U = \{h \in H \mid h x \in U\}.\]
    Since $U$ is open, the set $W \subseteq H$ is open.  Since $x \in U$ and $\Id_X \in H$, we have that $\Id_X \in W$. And finally, if $s \in S$ and $is \in W$, then $sx \in U$, whereby $i^{-1} W \subseteq R(x,U)$.
    
    Since the sets obtained as inverse images of open subsets containing zero of finite-dimensional tori under continuous homomorphisms form a basis for the neighborhoods of the identity in any compact Hausdorff abelian group, we have that there exists $d \in \N$, a non-empty, open set $V \subseteq \T^d$ containing $\vec 0$, and a homomorphism $\psi: H \to \T^d$ such that $\psi^{-1} V \subseteq W$.

    Considering $i$ as a semigroup homomorphism $S \to H$, define $\varphi = \psi \circ i: S \to \T^d$. It is a homomorphism, and
    \[\varphi^{-1} V = i^{-1} \psi^{-1} V \subseteq i^{-1} W \subseteq R(x,U) \subseteq A,\]
    as was to be shown.
\end{proof}

It is interesting to note that the minimality assumption in the second point in \cref{thm_equiv_char_of_bohr_zero_sets} is necessary.
Indeed, the action of $S$ in which all elements of $S$ send all points of a space $X$ to a fixed point $x_0 \in X$ is equicontinuous.
Without the minimality assumption, the empty set would satisfy the statement in the second point.
Minimality could, however, be replaced by the weaker assumption that every $s \in S$ act surjectively (cf. \cref{lemma_minimal_comm_action_implies_surjective}).

\subsubsection{Families of difference sets and sets of Bohr recurrence}
\label{sec_diff_sets_and_bz_sets}

In this subsection, we collect some of the facts that we need on the families of $\Delta$ sets and \bohrzero{} sets and their duals.  While these families were defined in the second paragraph of the introduction, we offer some clarification and equivalent forms here.  We keep the notation $\Delta$, $\Delta^*$, $\dcsyndeticbohr$, and $\dcthickbohr$ from the introduction.  Throughout, let $S$ be a semigroup.

The family of $\Delta^*$ sets is defined as the collection of sets that have non-empty intersection with every $\Delta$ set.  It is tempting but incorrect to guess that $A \in \Delta^*$ if and only if for all $s_1, s_2, \ldots \in S$, there exists $i, j \in \N$ with $i < j$ such that $s_j \in s_i A$.  
While this condition suffices, it is not necessary.
Indeed, the whole semigroup $S = (\N \cup \{0\})^{\N}$ is a $\Delta^*$ subset of itself, but if $s_i$ is the indicator function of $i \in \N$, then for all $i, j \in \N$ with $i < j$, $s_j \not\in s_i S$.  Thus, we stay with the definition of the family $\Delta^*$ as the dual of $\Delta$ and do not offer an alternative characterization.

We do, however, find it worthwhile to introduce and treat a family of subsets of $S$ related to $\Delta$ sets called $\Delta_0$ sets.  A set $A \subseteq S$ is a \emph{$\Delta_0$-set} if for all $N \in \N$, there exists $s_1, \ldots, s_N \in S$ such that for all $1 \leq i < j \leq N$, $s_j \in s_iA$.  Thus, if $\Delta$ sets contain the differences of an infinite set, $\Delta_0$ sets contain the differences of arbitrarily large finite sets.  It is immediate from the definitions that $\Delta \subseteq \Delta_0$.

In the introduction, we define a set of Bohr recurrence to be any set that has non-empty intersection with every \bohrzero{} set, that is, $\dcthickbohr = \dcsyndeticbohr^*$.  The following lemma justifies calling these ``sets of Bohr recurrence,'' at least in commutative semigroups.

\begin{theorem}
    \label{lemma_dcsyndeticbohr_dcthickbohr_dual}
    \lean{setOfBohrRecurrenceImpliesTimeOfRecForMinEqui}
    \lean{timeOfRecurrenceForMinEquiImpliesSetOfBohrRecurrence}
    Let $S$ be a semigroup and $A \subseteq S$. 
    The following are equivalent:
    \begin{enumerate}
        \item $A \in \dcsyndeticbohr^*$;
        \item for all $d \in \N$, all semigroup homomorphisms $\varphi: S \to \T^d$, and all open $U \subseteq \T^d$ containing $\vec 0$, there exists $a \in A$ such that $\varphi a \in U$;
        \item if $S$ is commutative, then for all minimal, equicontinuous $S$-systems $X$, all $x \in X$, and all open $U \subseteq X$ containing $x$, there exists $a \in A$ such that $a x \in U$.
    \end{enumerate}
\end{theorem}

\begin{proof}
\leanok{}
    That the first and second are equivalent is immediate from the definition of the dual in \eqref{eqn_fam_dual_def}.
    That the first and third are equivalent follows immediately from \cref{thm_equiv_char_of_bohr_zero_sets}.
\end{proof}

\begin{lemma}
    \label{lemma_dcsyndeticbohr_is_filter}
    \lean{bohrZeroFamilyIsFilter}
    \leanok{}
    Let $S$ be a semigroup. The family $\dcsyndeticbohr$ is a filter.
\end{lemma}

\begin{proof}
\leanok{}
    Note that $S$ is always a \bohrzero{} subset of itself by choosing in the definition the constant zero homomorphism from $S$ into $\T$.  Thus, $\dcsyndeticbohr$ is a non-empty family.
    To see that it is a filter, we must show that if $A, B \in \dcsyndeticbohr$, then $A \cap B \in \dcsyndeticbohr$.  Let $A, B \in \dcsyndeticbohr$.  There exist $k, \ell \in \N$, homomorphisms $\varphi: S \to \T^k$ and $\psi: S \to \T^\ell$, and open sets $U \subseteq \T^k$ and $V \subseteq \T^\ell$ containing $\vec 0$ such that $A \supseteq \varphi^{-1} U$ and $B \supseteq \psi^{-1} V$.
    
    The map $\varphi \otimes \psi: s \mapsto (\varphi(s), \psi(s)) \in \T^{k+\ell}$ is a homomorphism from $S$ into $\T^{k + \ell}$, and $U \times V \subseteq \T^{k + \ell}$ is an open set containing $\vec 0$.  Moreover, we see that
    \[A \cap B \supseteq \varphi^{-1} U \cap \psi^{-1} V = (\varphi \otimes \psi)^{-1} (U \times V),\]
    demonstrating that $A \cap B \in \dcsyndeticbohr$, as desired.
\end{proof}

The proof of the following lemma makes use of a future result, the well-known duality between filters and partition regular families, \cref{thm_family_pr_iff_dual_is_filter}.
The careful reader will note that that \cref{thm_family_pr_iff_dual_is_filter} does not depend in any way on \cref{lemma_dcthickbohr_is_pr}. 

\begin{lemma}
    \label{lemma_dcthickbohr_is_pr}
    \lean{setOfBohrRecurrenceFamilyIsPR}
    \leanok{}
    Let $S$ be a semigroup. The family $\dcthickbohr$ is partition regular.
\end{lemma}

\begin{proof}
\leanok{}
    That $\dcthickbohr$ is partition regular follows immediately from \cref{lemma_dcsyndeticbohr_is_filter} and the filter--partition regularity duality in \cref{thm_family_pr_iff_dual_is_filter}.
\end{proof}

\begin{theorem}
    \label{thm_delta_zero_set_of_bohr_rec}
    \label{deltaZeroImpliesSetOfBohrRecurrence}
    \leanok{}
    Let $S$ be a semigroup.  Then $\Delta_0 \subseteq \dcthickbohr$.
\end{theorem}

\begin{proof}
\leanok{}
    Let $A \in \Delta_0$ and $B \in \dcsyndetic$.  We must show that $A \cap B \neq \emptyset$.
    There exists $d \in \N$, a semigroup homomorphism $\varphi: S \to \T^d$, and an open set $U \subseteq \T^d$ containing $\vec 0$ such that $\varphi^{-1} U \subseteq B$.  It suffices to show $A \cap \varphi^{-1} U \neq \emptyset$; we will show, equivalently, that $\varphi A \cap U \neq \emptyset$.

    Let $\eps > 0$ be such that the Euclidean ball of radius $\eps$ centered at $\vec 0$ is contained in $U$.  Since $\T^d$ is totally bounded, there exists $N \in \N$ and Euclidean balls $B_1, \ldots, B_N \subseteq \T^d$ of diameter $\eps$ such that $\T^d = \cup_{i=1}^N B_i$.

    Since $A \in \Delta_0$, there exists $s_1, \ldots, s_{N+1} \in S$ such that for all $1 \leq i < j \leq N + 1$, $s_j \in s_i A$.  It follows that for all $1 \leq i < j \leq N + 1$, $\varphi s_j \in \varphi s_i + \varphi A$, that is, $\varphi s_j - \varphi s_i \in \varphi A$
    
    By the pigeonhole principle, there exists $1 \leq i < j \leq N + 1$ and $1 \leq k \leq N$ such that $\varphi s_i, \varphi s_j \in B_k$.  Since $B_k$ has diameter $\eps$, the distance between $\varphi s_j - \varphi s_i$ and $\vec 0$ is at most $\eps$. Therefore, $\varphi s_j - \varphi s_i \in U$. Since $\varphi s_j - \varphi s_i \in \varphi A$, we have demonstrated that $\varphi A \cap U \neq \emptyset$, as was to be shown.
\end{proof}

\subsection{Ultrafilters and dynamics}
\label{sec_ultrafilters}

Ultrafilters play a useful, though perhaps not entirely essential, role in this paper. In this subsection, we summarize just what we will need; the curious reader is directed to \cite{hindman_strauss_book_2012} for a comprehensive account.  As always, $S$ is a discrete semigroup.

We begin with a clarification on thick sets. We defined the family of thick sets, $\thick$, as $\syndetic^*$, where the dual star notation is defined in \eqref{eqn_fam_dual_def}.  It is more customary to define a set $A \subseteq S$ to be \emph{thick} if for all finite $F \subseteq S$, there exists $s \in S$ such that $Fs \subseteq A$.  It is a short exercise to show that these two definitions of thick agree.

\begin{lemma}
    \label{lemma_thick_inverse_shift_intersections}
    \label{lemma_preimage_intersection_of_thick_is_thick}
    \lean{inverseDilateCapOfThickIsThick}
    \label{lemma_thick_right_shift_is_thick}
    \lean{rightTransOfThickIsThick}
    \leanok{}
    Let $S$ be a semigroup and $H \subseteq S$ be thick.
    \begin{enumerate}
        \item For all finite $K \subseteq S$, the set $\cap_{k \in K} k^{-1} H$ is thick.

        \item For all $k \in S$, the set $Hk$ is thick.
    \end{enumerate}
\end{lemma}

\begin{proof}
\leanok{}
    1. Let $K \subseteq S$ be finite.  To see that $\cap_{k \in K} k^{-1} H$ is thick, let $F \subseteq S$ be finite.  The set $KF$ is finite, so there is $s \in S$ such that $KFs \subseteq H$. We see $Fs \subseteq \cap_{k \in K} k^{-1} H$, as desired.

    2. Let $k \in S$. Let $F \subseteq S$ be finite.  There exists $s \in S$ such that $Fs \subseteq H$.  Then $Fsk \subseteq Hk$.
\end{proof}

Recall from the introduction that an ultrafilter on $S$ is a family that is both partition regular and a filter.  We denote by $\beta S$ the collection of all ultrafilters on $S$.
The collection of sets of the form $\overline{A} \defeq \{p \in \beta S \mid A \in p\}$, where $A \subseteq S$, yields a clopen basis for a compact Hausdorff topology on $\beta S$.
The principal ultrafilter at $s \in S$ -- that is, the family of sets that contain $s$ -- is the image of the injection of $S$ into $\beta S$.
Thus $\overline{S} = \beta S$ is the Stone-\Cech{} compactification of $S$.

The semigroup operation $S \times S \to S$ extends to a operation $\beta S \times \beta S \to \beta S$ defined for $p, q \in \beta S$ by
\begin{align}
\label{eqn_ultra_mult}
    pq = \big\{ A \subseteq S \mid \big\{ s \in S \mid \{t \in S \mid st \in A \} \in q \big\} \in p \big\},
\end{align}
making $\beta S$ into a compact Hausdorff right-topological semigroup; see \cite[Ch. 4]{hindman_strauss_book_2012}. An ultrafilter $p \in \beta S$ is called \emph{idempotent} if $pp = p$ and \emph{minimal} if it belongs to a left ideal of $\beta S$ that is minimal among all left ideals of $\beta S$.  A \emph{minimal idempotent} is an idempotent ultrafilter that is minimal; such objects exist (see \cite[Cor. 2.6 and Thm. 2.9]{hindman_strauss_book_2012}) and are extraordinarily useful, as evidenced already in the proof of \cref{lemma_distal_implies_semisimple}.

Let $X$ be an $S$-system, $x \in X$, and $p, q \in \beta S$.
We denote by $px$ the ultrafilter limit of the pushforward of $p$ through the map $s \mapsto sx$.
This yields an action of $\beta S$ on $X$: $p(qx) = (pq)x$ (cf. \cite[Thm. 4.1 and Rmk. 4.3]{hindman_strauss_book_2012}).
If $p$ is idempotent, then $px$ is proximal to $x$; if $p$ is minimal, then $px$ is $S$-uniformly recurrent (cf. \cite[Thm. 19.26]{hindman_strauss_book_2012}).
It is quick to check from the definitions that if $U \subseteq X$ is open and $px \in U$, then $R(x,U) \in p$; as a partial converse, it is easy to see that if $R(x,U) \in p$, then $px \in \overline{U}$.

\begin{lemma}
    \label{lemma_thick_closure_contains_strong_ultrafilter}
    \lean{minIdempotentWitnessesShiftIntersectionLargeness}
    \leanok{}
    Let $S$ be a semigroup and $H \subseteq S$ be thick. There exists a minimal idempotent ultrafilter $p \in \overline{H}$ such that for all finite $F \subseteq S$, the set $\cap_{f \in F} f^{-1} H \in p$.
\end{lemma}

\begin{proof}
\leanok{}
    The set
    \[\big\{ \overline{H} \cap \bigcap_{s \in F} \overline{s^{-1} H} \mid \text{finite } F \subseteq S \big\}\]
    has the finite intersection property, whence
    \[L \defeq \overline{H} \cap \bigcap_{s \in S} \overline{s^{-1}H}\]
    is non-empty. It is quick to check that $L$ is a left ideal of $\beta S$.
    By \cite[Cor. 2.6 and Thm. 2.9]{hindman_strauss_book_2012}, there exists a minimal idempotent ultrafilter $p \in L \subseteq \overline{H}$.
    The stated property holds since for $F \subseteq S$ finite,
    \[p \in L \subseteq \bigcap_{f \in F} \overline{f^{-1} H}\]
    implies that $\cap_{f \in F} f^{-1} H \in p$, as desired.
\end{proof}

\section{Furstenberg family algebra}
\label{sec_algebra_structure_of_ff}

Let $S$ be a set and $\powerset(S)$ denote its powerset.  A collection $\family$ of subsets of $S$ is \emph{upward closed} if for all $A \in \family$ and all $B \subseteq S$ containing $A$, we have that $B \in \family$.  As defined in the introduction, a \emph{(Furstenberg) family} is an upward-closed collection of subsets of $S$.  We denote the collection of families by $\families(S) \subseteq \powerset(\powerset(S))$.

Throughout this section, we use the symbols $\family$, $\familytwo$, and $\familythree$ to denote families, ie. members of $\families(S)$.
We note that the collection of no sets $\emptyset$ and the collection of all sets $\powerset(S)$ are families.
This choice was made to complete the algebraic picture: many instances of the operations described below naturally yield $\emptyset$ or $\powerset(S)$.

\subsection{Family dual}

The \emph{family dual} map $*: \mathcal{P}(\mathcal{P}(S)) \to \mathcal{P}(\mathcal{P}(S))$ is defined by 
\[\family^* = \big\{A \subseteq S : \forall B \in \family, \ A \cap B \neq \emptyset \big\}.\]
We call $\family^*$ the \emph{dual} of $\family$.
As discussed in the introduction, this operation is well-known.  We summarize its main properties in the following theorem.

\begin{theorem}
    \label{thm_dual_main_properties}
    \lean{Family.famDual} \leanok{}
    \lean{dualIsInvolutionOnFamilies} \leanok{}
    \lean{dualIsAntitone} \leanok{}
    \lean{famDualAlt} \leanok{}
    The family dual, restricted to $\families(S)$,
    \begin{enumerate}
        \item is a unary operation on $\families(S)$: $\family^* \in \families(S)$;
        
        \item is an involution of $\families(S)$: $(\family^*)^* = \family$;
        
        \item is antitone: if $\family \subseteq \familytwo$, then $\familytwo^* \subseteq \family^*$;

        \item is described equivalently by:
    \[\family^* = \big\{A \subseteq S : S \setminus A \not\in \family\big\}.\]
    \end{enumerate}
\end{theorem}

\begin{proof}
All four statements are simple exercises from the definitions and can be found in \cite[Ch. 2]{akin_1997}.
\end{proof}

It follows formally from the definition but is worth making explicit: the families $\emptyset$ and $\powerset(S)$ are dual: $\emptyset^* = \powerset(S)$ and $\powerset(S)^* = \emptyset$.

\subsection{Family meet}
\label{sec_family_meet}

The \emph{family meet} map $\classcapdual: \mathcal{P}(\mathcal{P}(S))^2 \to \mathcal{P}(\mathcal{P}(S))$ is defined by
\[\family \classcapdual \familytwo = \big\{A \subseteq S : \forall B \in \family^*, \ A \cap B \in \familytwo \big\}.\]
We call $\family \classcapdual \familytwo$ the \emph{family meet} of $\family$ and $\familytwo$.
The following theorem summarizes the basic properties of family meet.  Of particular note is the fact that the operation is commutative.  While not difficult to prove, it is not apparent from the definition.

\begin{theorem}
\label{thm_meet_algebra_properties}
\label{thm_classcapdual_monotone}
\label{thm_classcapdual_is_monotonic}
\lean{Family.famMeet} \leanok{}
\lean{familyMeetIsAssociative} \leanok{}
\lean{familyMeetIsCommutative} \leanok{}
\lean{familyMeetIsMonotonic} \leanok{}
    The family meet map, restricted to $\families(S)^2$,
    \begin{enumerate}
        \item is a binary operator on $\families(S)$: for all $\family, \familytwo \in \families(S)$, $\family \classcapdual \familytwo \in \families(S)$;

        \item is associative: $(\family \classcapdual \familytwo) \classcapdual \familythree = \family \classcapdual (\familytwo \classcapdual \familythree)$;

        \item is commutative: $\family \classcapdual \familytwo = \familytwo \classcapdual \family$;

        \item is monotonic: if $\family_1 \subseteq \family_2$ and $\familytwo_1 \subseteq \familytwo_2$, then $\family_1 \classcapdual \familytwo_1 \subseteq \family_2 \classcapdual \familytwo_2$.
    \end{enumerate}
\end{theorem}

\begin{proof}
If $A \in \family \classcapdual \familytwo$ and $A \subseteq B$, then $B \in \family \classcapdual \familytwo$ because $\familytwo$ is upward closed, proving 1.

Associativity of $\classcapdual$ can be proved by a straightforward but tedious argument from the definitions.  A shorter argument is available via associativity of the family join operator, the fact that the dual is an involution on families, and the DeMorgan-type laws in \cref{thm_de_morgan_v1}, none of which depend on associativity of the family meet operator. Indeed, taking the dual of statement 2. in \cref{thm_join_properties} applied to the families $\family^*$, $\familytwo^*$, and $\familythree^*$, we see
\[((\family^* \classcap \familytwo^*) \classcap \familythree^*)^* = (\family^* \classcap (\familytwo^* \classcap \familythree^*))^*,\]
which simplifies to associativity of $\classcapdual$ by \cref{thm_de_morgan_v1} and the fact that the family dual is an involution.

While commutativity of $\classcapdual$ can be shown in the same way as associativity using the DeMorgan-type laws, a direct argument is short.  Let $A \in \family \classcapdual \familytwo$. To see that $A \in \familytwo \classcapdual \family$, let $B \in \familytwo^*$.  Since $S \setminus B \not\in \familytwo$, we have that $A \cap ((S \setminus A) \cup (S \setminus B)) \not\in \familytwo$.  Since $A \in \family \classcapdual \familytwo$, we see that $(S \setminus A) \cup (S \setminus B) \not\in \family^*$, whence $A \cap B \in \family$, as desired.  We have shown that $\family \classcapdual \familytwo \subseteq \familytwo \classcapdual \family$.  The reverse containment holds by swapping $\family$ and $\familytwo$.

Monotonicity follows from a short exercise using the definition of family meet and the fact that the dual is antitone.  
\end{proof}

The following theorem summarizes the cases in which the family meet yields $\emptyset$ or $\powerset(S)$.

\begin{theorem}
\label{thm_familymeet_containment_conditions}
\label{thm_classcapdual_empty_condition}
\label{thm_classcapdual_contained_in_intersection}
\lean{familyMeetContainedInIntersection} \leanok{}
\lean{fullFamilyMeetCondition} \leanok{}
\lean{emptyFamilyMeetCondition} \leanok{}
    For all $\family, \familytwo \in \families(S)$,
    \begin{enumerate}
        \item \label{item_meet_contained_in_family}
        if $\familytwo \neq \powerset(S)$, then $\family \classcapdual \familytwo \subseteq \family$;

        \item $\family \classcapdual \familytwo = \powerset(S)$ if and only if $\family = \powerset(S)$ or $\familytwo = \powerset(S)$; and

        \item \label{item_meet_empty_condition}
        $\family \classcapdual \familytwo = \emptyset$ if and only if $\family^* \not\subseteq \familytwo$.
    \end{enumerate}
\end{theorem}

\begin{proof}
    1. \ Suppose $\familytwo \neq \powerset(S)$.  Since $(\family^*)^* = \family$, it suffices to show that $\family \classcapdual \familytwo \subseteq (\family^*)^*$.  Let $A \in \family \classcapdual \familytwo$ and $B \in \family^*$.  Then $A \cap B \in \familytwo$.  Since $\familytwo \neq \powerset(S)$, $A \cap B \neq \emptyset$, showing that $A \in (\family^*)^*$.

    2. \ It is immediate from the definition of family meet that $\family \classcapdual \powerset(S) = \powerset(S)$.  By commutativity, it follows that if $\family = \powerset(S)$ or $\familytwo = \powerset(S)$, then $\family \classcapdual \familytwo = \powerset(S)$.  Conversely, if $\family \neq \powerset(S)$ and $\familytwo \neq \powerset(S)$, it follows from 1. that $\family \classcapdual \familytwo \subseteq \family \neq \powerset(S)$, as desired.

    3. \ Since $\family \classcapdual \familytwo \in \families(S)$, we have that $\family \classcapdual \familytwo = \emptyset$ if and only if $S \not\in \family \classcapdual \familytwo$.  By definition, $S \not\in \family \classcapdual \familytwo$ if and only if there exists $B \in \family^*$ such that $B \not\in \familytwo$, that is, if and only if $\family^* \not\subseteq \familytwo$.
\end{proof}

\subsection{Family join}

The \emph{family join} map $\classcap: \mathcal{P}(\mathcal{P}(S))^2 \to \mathcal{P}(\mathcal{P}(S))$ is defined by
\[\family \classcap \familytwo = \big\{C \subseteq S : \exists A \in \family, \ \exists B \in \familytwo, \ A \cap B = C \big\}\]
We call $\family \classcap \familytwo$ the \emph{family join} of $\family$ and $\familytwo$.

As discussed in the introduction, the family join map appears in the literature and is usually denoted $\family \cdot \familytwo$.  We depart from this notation since the product of two families has, to us, a more natural meaning when $S$ is a semigroup (see \cref{sec_family_meet_with_algebra}).  We choose to write family join as $\classcap$ since this symbol has elements of both the union symbol (it is a type of join) and the intersection symbol (the family is described in terms of intersections of sets).

\begin{theorem}
\label{thm_join_properties}
\lean{Family.famJoin} \leanok{}
\lean{familyJoinIsAssociative} \leanok{}
\lean{familyJoinIsCommutative} \leanok{}
\lean{familyJoinIsMonotonic} \leanok{}
    The family join, restricted to $\families(S)^2$,
    \begin{enumerate}
        \item is a binary operator on $\families(S)$: for all $\family, \familytwo \in \families(S)$, $\family \classcap \familytwo \in \families(S)$;

        \item is associative: $(\family \classcap \familytwo) \classcap \familythree = \family \classcap (\familytwo \classcap \familythree)$;

        \item is commutative: $\family \classcap \familytwo = \familytwo \classcap \family$;

        \item is monotonic: if $\family_1 \subseteq \family_2$ and $\familytwo_1 \subseteq \familytwo_2$, then $\family_1 \classcap \familytwo_1 \subseteq \family_2 \classcap \familytwo_2$.
    \end{enumerate}
\end{theorem}

\begin{proof}
All four statements are simple exercises from the definitions and can be found in \cite[Ch. 2]{akin_1997}.    
\end{proof}

The following theorem summarizes the cases in which the family join yields $\emptyset$ or $\powerset(S)$.

\begin{theorem}
\label{thm_familyjoin_containment_conditions}
\label{thm_classcap_all_sets_condition}
\label{thm_classcap_larger}
\lean{familyContainedInFamilyJoin} \leanok{}
\lean{fullFamilyJoinCondition}
\lean{emptyFamilyJoinCondition} \leanok{}
    For all $\family, \familytwo \in \families(S)$,
    \begin{enumerate}
        \item \label{item_join_contains_family}
        if $\familytwo \neq \emptyset$, then $\family \subseteq \family \classcap \familytwo$;
        
        \item $\family \classcap \familytwo = \emptyset$ if and only if $\family = \emptyset$ or $\familytwo = \emptyset$; and

        \item $\family \classcap \familytwo = \mathcal{P}(S)$ if and only if $\family \not\subseteq \familytwo^*$.
    \end{enumerate}
\end{theorem}

\begin{proof}
    The proofs of all three statements are simple analogues of those in the proof of \cref{thm_familymeet_containment_conditions} and so are left to the reader.
\end{proof}

\subsection{DeMorgan-type laws}

The DeMorgan-type laws expressed in the following theorem relate family dual, meet, and join as would be expected by analogy with set complement, intersection, and union.

\begin{theorem}
    \label{thm_de_morgan_v1}
    \lean{familyDeMorgan1}
    \lean{familyDeMorgan2}
    For all $\family, \familytwo \in \families(S)$,
    \begin{enumerate}
        \item $(\family \classcap \familytwo)^* = \family^* \classcapdual \familytwo^*$, and
        
        \item $(\family \classcapdual \familytwo)^* = \family^* \classcap \familytwo^*$.
    \end{enumerate}
\end{theorem}

\begin{proof}
    To see 1., note that in the following list, each statement is equivalent to the one that follows it.
    \begin{enumerate}[label=\roman*.]
        \item $A \in (\family \classcap \familytwo)^*$.
        \item for all $B \in \family \classcap \familytwo$, $A \cap B \neq \emptyset$.
        \item for all $C \in \family$ and $D \in \familytwo$, $A \cap C \cap D \neq \emptyset$.
        \item for all $C \in \family$, $A \cap C \in \familytwo^*$.
        \item for all $C \in (\family^*)^*$, $A \cap C \in \familytwo^*$.
        \item $A \in \family^* \classcapdual \familytwo^*$.
    \end{enumerate}
    Each equivalence follows directly from the definitions, except for the equivalence between iv. and v., which follows from the fact that the family dual is an involution on families.

    Now 2. follows from analogous reasoning: in the following list, each statement is equivalent to the one that follows it.
    \begin{enumerate}[label=\roman*.]
        \item $A \in (\family \classcapdual \familytwo)^*$.
        \item for all $B \in \family \classcapdual \familytwo$, $A \cap B \neq \emptyset$.
        \item for all $B \in (\family^*)^* \classcapdual (\familytwo^*)^*$, $A \cap B \neq \emptyset$.
        \item for all $B \in (\family^* \classcap \familytwo^*)^*$, $A \cap B \neq \emptyset$.
        \item $A \in ((\family^* \classcap \familytwo^*)^*)^*$.
        \item $A \in \family^* \classcap \familytwo^*$.
    \end{enumerate}
    Each equivalence follows directly from the definitions, except for the equivalences between ii. and iii. and between v. and vi., which follow from the fact that the family dual is an involution on families, and the equivalence between iii. and iv., which follows from the first paragraph of this proof.
\end{proof}

It is immediate that the union and intersection of a collection of families is again a family.  The DeMorgan-type laws apply to unions and intersections, as well.

\begin{lemma}
    \label{lemma_demorgan_for_union_and_int}
    \lean{unionDeMorgan1}
    \lean{unionDeMorgan2}
    For all collections $(\family_i)_{i \in I}$ of families of subsets of $S$,
    \begin{enumerate}
        \item $\displaystyle \bigg( \bigcap_i \family_i \bigg)^* = \bigcup_i \family_i^*$, and
        
        \item $\displaystyle \bigg( \bigcup_i \family_i \bigg)^* = \bigcap_i \family_i^*$.
    \end{enumerate}
\end{lemma}

\begin{proof}
    By \cref{thm_dual_main_properties}, each of the following statements is equivalent to the one that follows it.
    \begin{enumerate}[label=\roman*.]
        \item $A \in \left( \cap_i \family_i \right)^*$.
        \item $S \setminus A \not\in \cap_i \family_i$.
        \item There exists $i \in I$, $S \setminus A \not\in \family_i$.
        \item There exists $i \in I$, $A \in \family_i^*$.
        \item $A \in \cup_i \family_i^*$.
    \end{enumerate}
    The proof of 2. is analogous.
\end{proof}

The following lemma helps to reformulate quantifier strings like those found in the statement of \cref{mainthm_gen_of_bfw} into the family algebra introduced above.  For a family $\familytwo$ and a set $F \subseteq S$, we denote by $\familytwo_{\cap F}$ the family of sets $A \subseteq S$ for which $A \cap F \in \familytwo$.

\begin{lemma}
\label{lemma_desc_of_very_strong}
\lean{memberOfCapFamJoinH}
\lean{memberOfCapFamDualMeetH}
\lean{iInterCapFamilyDescription}
\lean{iUnionCapFamilyDualDescription}
    For all $\familytwo, \familythree \in \families(S)$ and $A, F \subseteq S$,
    \begin{enumerate}
        \item $A \in \familytwo_{\cap F} \classcap \familythree$ if and only if there exists $G \in \familytwo$ with $G \subseteq F$ and $B \in \familythree$ such that $A \cap G = B \cap G$; and
        
        \item $A \in (\familytwo_{\cap F})^* \classcapdual \familythree$ if and only if for all $G \in \familytwo$ with $G \subseteq F$, $A \cap G \in \familythree$.
    \end{enumerate}
    Consequently, for all $\family, \familytwo, \familythree \in \families(S)$,
    \begin{enumerate}
        \setcounter{enumi}{2}
        \item $\displaystyle \bigcap_{F \in \family} \big( \familytwo_{\cap F} \classcap \familythree \big) = \big\{ A \subseteq S : \forall F \in \family, \ \exists G \in \familytwo, \ G \subseteq F, \ \exists B \in \familythree, \ A \cap G = B \cap G \big\}$, and

        \item $\displaystyle \bigcup_{F \in \family} \big( (\familytwo_{\cap F})^* \classcapdual \familythree \big) = \big\{ A \subseteq S : \exists F \in \family, \ \forall G \in \familytwo, \ G \subseteq F, \ A \cap G \in \familythree \big\}$.
    \end{enumerate}
\end{lemma}

\begin{proof}
    To show 1., by the definition of family join and $\familytwo_{\cap F}$, we must show that the following are equivalent:
    \begin{enumerate}[label=\roman*.]
        \item there exists $G \subseteq S$ for which $G \cap F \in \familytwo$ and $B \in \familythree$ such that $A = B \cap G$;

        \item there exists $G \in \familytwo$ with $G \subseteq F$ and $B \in \familythree$ such that $A \cap G = B \cap G$.
    \end{enumerate}
    If i. holds, then $G' \defeq G \cap F \in \familytwo$, and $A \cap G' = B \cap G'$.  If ii. holds, then we see
    \begin{align*}
        A &= \big(A \cap (S \setminus G) \big) \cup \big(A \cap G \big) \\
        &= \big(A \cap (S \setminus G) \big) \cup \big(B \cap G \big)\\
        &= \big( G \cup ( A \cap (S \setminus G)) \big) \cap \big( B \cup (A \cap (S \setminus G))\big) \\
        &= G' \cap B',
    \end{align*}
    which shows that $A$ is the intersection of a set $G' \subseteq S$ for which $G' \cap F \in \familytwo$ (since it contains $G$) and $B' \in \familythree$ (since it contains $B$).

    The proof of 2. is similar but simpler, and so we leave it to the reader. The statements in 3. and 4. follow immediately from those in 1. and 2.
\end{proof}

\subsection{Filters and partition regular families}
\label{sec_filters_and_pr_fams_algebra}

As defined in the introduction, a family $\family$ is a \emph{filter} if it is non-empty and for all $A, B \in \family$, we have that $A \cap B \in \family$.  A family $\family$ is \emph{partition regular} if it is not equal to $\powerset(S)$ and for all $A, B \subseteq S$, that $A \cup B \in \family$ implies that $A \in \family$ or $B \in \family$.

Let us explain our conventions.
In the literature, filters are always defined to be non-empty and are often defined to not include the empty set.
We have decided to adopt the first convention but not the second.
To maintain the duality between filters and partition regularity (as in \cref{thm_pr_iff_dual_is_filter}), we are led to the somewhat strange choice that the empty family $\emptyset$ is partition regular while the power set $\powerset(S)$ is not.
These specific choices are not particularly important for our purposes here: all interesting applications will feature non-empty, non-power-set families, and
all results in this section remain true or require only minor modifications in the event that $\emptyset$ is allowed to be filter or $\powerset(S)$ is not.

The following theorem characterizes families that are idempotent for meet or join, showing, in particular, that $(\families(S),\classcap,\classcapdual)$ is far from being a lattice.  The statement about filters is contained in \cite[Prop. 2.1(d)]{akin_1997}.

\begin{theorem}
\lean{familyIsIdempotentForJoinIffFilter}
\lean{familyIsIdempotentForMeetIffPR}
    Let $\family \in \families(S)$.
    \begin{enumerate}
        \item The family $\family$ is a filter or $\family = \emptyset$ if and only if $\family \classcap \family = \family$.
        
        \item The family $\family$ is partition regular or $\family = \powerset(S)$ if and only if $\family \classcapdual \family = \family$.
    \end{enumerate}
\end{theorem}

\begin{proof}
    1. \ If $\family = \emptyset$ or $\family$ is a filter, then it follows by the definitions that $\family \classcap \family = \family$.  Conversely, if $\family \classcap \family = \family$ and $\family \neq \emptyset$, then $\family$ is a filter: it is non-empty and for all $A, B \in \family$, $A \cap B \in \family \classcap \family = \family$.

    2. \ Just as 1., this can be verified by an elementary argument from the definitions.  A shorter argument is available using \cref{thm_pr_iff_dual_is_filter}, the fact that the dual is an involution, and the DeMorgan-type laws.  Indeed, the family $\family$ is partition regular or $\family = \powerset(S)$ if and only if the family $\family^*$ is a filter or $\family^* = \emptyset$. By 1., this happens if and only if $\family^* \classcap \family^* = \family^*$, which, by taking the dual, is seen to be equivalent to $\family \classcapdual \family = \family$ by the DeMorgan-type laws and the fact that the family dual is an involution.
\end{proof}

The duality between partition regular families and filters expressed in the following theorem is well-known; see, eg. \cite[Pg. 26]{akin_1997} and \cite[Prop. 2.5(e)]{christopherson_johnson_2022}.

\begin{theorem}
\label{thm_pr_iff_dual_is_filter}
\label{thm_family_pr_iff_dual_is_filter}
\lean{familyIsPRIffDualIsFilter}
\lean{dualFFilterIffFMeetDualFIsDualF}
\lean{FJoinDualFIsFIffFMeetDualFIsDualF}
\lean{}
\lean{extendedFilterPRDuality}
    Let $\family \in \families(S)$.  The following are equivalent:
    \begin{enumerate}
        \item $\family$ is partition regular
        \item $\family^*$ is a filter
        \item $\family \classcap \family^* = \family$
        \item $\family \classcapdual \family^* = \family^*$
    \end{enumerate}
    If one (equivalently, all) of these hold and if $\family \neq \emptyset$, then $\family^* \subseteq \family$.
\end{theorem}

\begin{proof}
We begin by showing that $\family$ is partition regular if and only if $\family^*$ is a filter.  Each statement in the following list is equivalent to the one that precedes it by the reason given between the statements.
    \begin{enumerate}[label=\roman*.]
        \item $\family$ is a partition regular family.
        \begin{itemize}
            \item[--] By definition and the fact that $\powerset(S)^* = \emptyset$.
        \end{itemize}
    
        \item $\family^* \neq \emptyset$ and for all $A, B \subseteq S$, if $A \cup B \in \family$, then $A \in \family$ or $B \in \family$.
        \begin{itemize}
            \item[--] The family dual is an involution.
        \end{itemize}
    
        \item $\family^* \neq \emptyset$ and for all $A, B \subseteq S$, if $A \cup B \in (\family^*)^*$, then $A \in (\family^*)^*$ or $B \in (\family^*)^*$.
        \begin{itemize}
            \item[--] By the equivalent description of the family dual.
        \end{itemize}
    
        \item $\family^* \neq \emptyset$ and for all $A, B \subseteq S$, if $S \setminus (A \cup B) \not\in \family^*$, then $S \setminus A \not\in \family^*$ or $S \setminus B \not\in \family^*$.
        \begin{itemize}
            \item[--] By basic set theory.
        \end{itemize}
    
        \item $\family^* \neq \emptyset$ and for all $A, B \subseteq S$, if $(S \setminus A) \cap (S \setminus B) \not\in \family^*$, then $S \setminus A \not\in \family^*$ or $S \setminus B \not\in \family^*$.
        \begin{itemize}
            \item[--] By basic logic: $(\neg p \Rightarrow (\neg q \vee \neg r)) \iff (q \wedge r \to p)$.
        \end{itemize}
    
        \item $\family^* \neq \emptyset$ and for all $A, B \subseteq S$, if $S \setminus A \in \family^*$ and $S \setminus B \in \family^*$, then $(S \setminus A) \cap (S \setminus B) \in \family^*$.
        \begin{itemize}
            \item[--] Label $C = S \setminus A$ and $D = S \setminus B$ and use basic set theory.
        \end{itemize}
    
        \item $\family^* \neq \emptyset$ and  for all $C, D \subseteq S$, if $C \in \family^*$ and $D \in \family^*$, then $C \cap D \in \family^*$.
        \begin{itemize}
            \item[--] By definition.
        \end{itemize}
    
        \item $\family^*$ is a filter.
    \end{enumerate}

    That 3. and 4. are equivalent follows immediately from the DeMorgan-type laws.  To conclude that 1. through 4. are equivalent, it suffices to show, then, that $\family^*$ is a filter if and only if $\family \classcapdual \family^* = \family^*$.

    Suppose $\family^*$ is a filter.  Since $\family^* \neq \emptyset$, we have that $\family \neq \powerset(S)$, and so by \cref{thm_familymeet_containment_conditions} and commutativity of family meet that $\family \classcapdual \family^* \subseteq \family^*$.  That $\family^* \subseteq \family \classcapdual \family^*$ follows then immediately from the fact that $\family^*$ is a filter and the definition of family meet.

    Conversely, suppose that $\family^* = \family \classcapdual \family^*$.  It follows from statement 3. in \cref{thm_familymeet_containment_conditions} that $\family \classcapdual \family^* \neq \emptyset$, so $\family^* \neq \emptyset$. If $A, B \in \family^*$, then we see from $A \in \family^* = \family \classcapdual \family^*$ that $A \cap B \in \family^*$, showing that $\family^*$ is a filter.

    Finally, assume one of the 4 equivalent statements hold and that $\family \neq \emptyset$.  Since $\family^* \neq \powerset(S)$, we have by \cref{thm_familymeet_containment_conditions} that $\family^* = \family \classcapdual \family^* \subseteq \family$, as was to be shown.
\end{proof}

That $\family \classcap \family^*$ is partition regular is well-known \cite[Prop. 2.5 (h)]{christopherson_johnson_2022}.  The following theorem states both this fact and its dual formulation, which, in our opinion, is significantly more natural to prove.  The fact, then, that $\family \classcap \family^*$ is partition regular follows immediately from \cref{thm_pr_iff_dual_is_filter}.
While family meet appears to be new, this dual approach can be found elsewhere, eg. \cite[Prop. 2.1(e)]{akin_1997}.

\begin{theorem}
\label{thm_fam_classcap_dual_is_pr}
\lean{meetOfFamilyAndDualIsFilter}
\lean{joinOfFamilyAndDualIsPR}
    For all $\family \in \families(S)$,
    \begin{enumerate}
        \item the family $\family \classcapdual \family^*$ is a filter, and
                
        \item the family $\family \classcap \family^*$ is partition regular.
    \end{enumerate}
\end{theorem}

\begin{proof}
1. \ If follows from statement 3. in \cref{thm_familymeet_containment_conditions} that $\family \classcapdual \family^* \neq \emptyset$.  Suppose $A, B \in \family \classcapdual \family^*$.  To see that $A \cap B \in \family \classcapdual \family^*$, let $C \in \family^*$.   Then $B \cap C \in \family^*$, so $A \cap B \cap C \in \family^*$, as was to be shown.

2. \ That $\family \classcap \family^*$ is partition regular follows immediately from the DeMorgan-type laws and the filter-partition regular duality in \cref{thm_pr_iff_dual_is_filter}.
\end{proof}

\section{Recitation of the algebra}
\label{sec_recitation}

As demonstrated in the introduction, we use the algebra of Furstenberg families detailed in \cref{sec_algebra_structure_of_ff} in the course of our proof of \cref{mainthm_gen_of_bfw}.
The goal in this section is to give further concrete examples of the utility of that algebra.
More precisely,
\begin{itemize}
    \item in \cref{sec_ps_is_partition_reg}, we give the shortest proof we know of that piecewise syndeticity is a partition regular notion;

    \item in \cref{sec_char_of_central_star}, we give two new characterizations of members of the family $\central^*$, sets that have non-empty intersection with all central sets; and

    \item in \cref{sec_ps_and_strongps_families}, we give an algebraic description of the ``strongly piecewise'' construct.

\end{itemize}
Throughout this section, let $S$ be a semigroup.

\subsection{Piecewise syndetic and thickly syndetic sets}
\label{sec_ps_is_partition_reg}

Recall that $\syndetic$ and $\thick$ denote the families of syndetic and thick subsets of $S$, respectively. The family of \emph{piecewise syndetic} sets is defined to be
\[\PS = \syndetic \classcap \thick.\]
That $\PS$ is partition regular was shown first in \cite{brown_1971} and has appeared numerous times since.  We offer our brief take here via thickly syndetic sets and the algebra in \cref{sec_algebra_structure_of_ff}.

The dual of $\PS$ is $\PS^*$, the family of \emph{thickly syndetic}.  Via the DeMorgan-type laws in \cref{thm_de_morgan_v1},
\[\PS^* = \thick \classcapdual \syndetic,\]
which is to say that a set is thickly syndetic if and only if it has syndetic intersection with every syndetic set.  It can immediately be seen that $\PS^*$ is a filter: if two sets have syndetic intersection with every syndetic set, then so does their intersection.  By the filter-partition regularity duality in \cref{thm_pr_iff_dual_is_filter}, it follows that piecewise syndeticity is partition regular.

\subsection{The family of central sets and its dual}
\label{sec_char_of_central_star}

The algebra in \cref{sec_algebra_structure_of_ff} provides some new and interesting characterizations of the family of sets that have non-empty intersection with all central sets.  To get there, we begin by defining central sets and proving a folklore characterization of them in terms of dynamically central syndetic sets.
We will use ultrafilters in this subsection and the following ones; recall the preliminaries from \cref{sec_ultrafilters}.

Central sets are class of combinatorially rich sets introduced by Furstenberg \cite[Def. 8.3]{furstenberg_book_1981} to aid in treating problems in Ramsey Theory and additive combinatorics with tools from topological dynamics.  We define a set $A \subseteq S$ to be \emph{central} if there exists an $S$-system $X$, a point $x \in X$, and an open set $U \subseteq X$ containing a point that is both proximal to $x$ and $S$-uniformly recurrent for which $A = R(x,U)$.  We denote by $\central$ the collection of central subsets of $S$.

There are two striking facts about the collection of central sets that are not at all apparent from the definition: 1) $\central$ is a Furstenberg family (it is upward-closed), and 2) $\central$ is partition regular.  Both of these facts follow from a well-known characterization of central sets as precisely those sets which belong to some minimal, idempotent ultrafilter on $S$ \cite[Thm. 19.27]{hindman_strauss_book_2012}.

The family of central sets is closely related to the family of dynamically central syndetic sets.  A set $A \subseteq S$ is \emph{dynamically central syndetic} if there exists a minimal $S$-system $X$, a point $x \in X$, and an open set $U \subseteq X$ containing $x$ such that $R(x,U) \subseteq A$.
This is a family of syndetic subsets of $S$ that contain a global, highly structured subset; the name comes from \cite{glasscock_le_synd_idem_filters_2025}.
Dual to $\dcsyndetic$ is the family $\dcthick$, the family of \emph{sets of pointwise recurrence}, that is, those sets $A \subseteq S$ with the property that for all minimal $S$-systems $X$, all $x \in X$, and all open $U \subseteq X$ containing $x$, there exists $a \in A$ such that $ax \in U$.

We will have need of the following characterization of central sets, shown for $S = (\N,+)$ in \cite[Thm. 3.7]{Huang_Shao_Ye_2019}.
The result in greater generality is folklore; since we are not aware of a convenient reference, we give a complete proof.

\begin{theorem}
\label{thm_central_is_thick_classcap_dcsyndetic}
\lean{centralIsdcSCapThick}
\leanok{}
    Let $S$ be a semigroup.  A set $A \subseteq S$ is central if and only if it is the intersection of a dynamically central syndetic set and a thick set:
    \[\central = \dcsyndetic \classcap \thick.\]
\end{theorem}

\begin{proof}
\leanok{}
    Suppose $A \in \central$ so that, according to the definition, there is an $S$-system $X$, points $x, y \in X$ with $x$ proximal to $y$, the point $y$ $S$-uniformly recurrent, and an open set $U \subseteq X$ containing $y$ such that $A = R(x,U)$.  By \cref{lemma_uniform_near_diag}, there exists an open set $\alpha \subseteq X^2$ containing the diagonal and an open set $V \subseteq X$ containing $y$ such that
    \[R(y,V) \cap R\big((x,y), \alpha\big) \subseteq R(x,U) = A.\]
    Because $y$ is $S$-uniformly recurrent and $y \in V$ and $V$ is open, the set $R(y,V)$ is dynamically central syndetic. Because $x$ and $y$ are proximal, the set $R\big((x,y), \alpha \big)$ is thick.  Therefore, we see that $A \in \dcsyndetic \classcap \thick$, as desired.

    Conversely, suppose $A \in \dcsyndetic \classcap \thick$.  We will prove that there exists a minimal idempotent ultrafilter $p \in \beta S$ such that $A \in p$.  It will follow from \cite[Thm. 19.27]{hindman_strauss_book_2012} that the set $A$ is central according to our definition above.

    There exists a minimal $S$-system $X$, a point $x \in X$, an open set $U \subseteq X$ containing $x$, and a thick set $H \subseteq S$ such that $A \supseteq R(x,U) \cap H$.  Define $\mathscr{F}$ to be the upward closure of all sets of the form $R(x,V)$, where $V \subseteq X$ is open and contains $x$.
    It is easy to see that $\mathscr{F}$ is a filter every member of which is syndetic.
    
    Define
    \[\overline{\mathscr{F}} \defeq \bigcap_{C \in \mathscr{F}} \overline{C}.\]
    We claim that $\overline{\mathscr{F}}$ is a compact subsemigroup of $\beta S$. Indeed, it is compact since it is closed and $\beta S$ is Hausdorff.  To see that it is a semigroup, suppose $p, q \in \overline{\mathscr{F}}$. To see that $pq \in \overline{\mathscr{F}}$, it suffices to show that for all open $V \subseteq X$ containing $x$, we have that $R(x, V) \in pq$.  Let $V \subseteq X$ be open and contain $x$.
    Note that since $q \in \overline{\mathscr{F}}$, we have that if $s \in R(x,V)$, then $R(x,s^{-1}V) \in q$, whereby
    \begin{align*}
        R(x,V) &\subseteq \big\{ s \in S \mid R(x,s^{-1}V) \in q \} \\
        &= \big\{ s \in S \mid s^{-1}R(x,V) \in q \}\\
        &= \big\{ s \in S \mid \{ t \in T \mid st \in R(x,V) \} \in q \}.
    \end{align*}
    Because $p \in \overline{\mathscr{F}}$, we have that $R(x,V) \in p$, so
    \[\big\{ s \in S \mid \{ t \in T \mid st \in R(x,V) \} \in q \} \in p,\]
    which is to say that $R(x, V) \in pq$, as desired.

    Since $H$ is thick, by \cite[Thm. 4.48]{hindman_strauss_book_2012}, there exists a minimal left ideal $L \subseteq \overline{H}$.
    We claim that the set $L \cap \overline{\mathscr{F}}$ contains an idempotent.
    We will use the fact that for all syndetic sets $A \subseteq S$ and all minimal left ideals $L \subseteq \beta \N$, we have $\overline{A} \cap L \neq \emptyset$ (\cite[Thm. 4.48]{hindman_strauss_book_2012}).  Since $L$ is a minimal left ideal, it is compact \cite[Thm. 2.6]{hindman_strauss_book_2012}.  Since $\mathscr{F}$ is a filter and every member is syndetic, the family
    \[\big\{\overline{C} \cap L  \mid  C \in \mathscr{F} \big\}\]
    has the finite intersection property.  Since $L$ is a compact subsemigroup, $\overline{\mathscr{F}} \cap L$ is a non-empty, compact subsemigroup. By \cite[Thm. 2.5]{hindman_strauss_book_2012}, there exists an idempotent $q \in L \cap \overline{\mathscr{F}}$, as claimed.

    Finally, note that
    \[q \in L \cap \overline{\mathscr{F}} \subseteq \overline{H} \cap \overline{R(x,U)},\]
    whereby the set $R(x,U) \cap H$ is contained in the minimal idempotent $q$, as was to be shown.
\end{proof}

From the characterization of central sets in \cref{thm_central_is_thick_classcap_dcsyndetic} and the DeMorgan-type laws in \cref{thm_de_morgan_v1}, we see that
    \begin{align}
        \label{eqn_central_star_formula}
        \central^* = \syndetic \classcapdual \dcthick.
    \end{align}
This leads immediately to the following two characterizations of $\central^*$ sets, both of which, we believe, are new.  The characterization in \cref{thm_chars_of_central_star} number 3. is a combinatorial relative and apparent strengthening of a classic theorem of Furstenberg \cite[Thm. 9.11]{furstenberg_book_1981}: a point $x$ in a $\Z$-system $X$ is $\central^*$-recurrent (that is, for all open $U \subseteq X$ containing $x$, the set $R(x,U) \in \central^*$) if and only if for all $\Z$-systems $Y$ and all uniformly recurrent points $y \in Y$, the point $(x,y)$ is uniformly recurrent in the product system $X \times Y$.

\begin{theorem}
\label{thm_chars_of_central_star}
\lean{cStarIffSetOfRecAlongAllThick}
\lean{cStarIffSyndeticAlongdcS}
\leanok{}
    Let $S$ be a semigroup and $A \subseteq S$.  The following are equivalent:
    \begin{enumerate}
        \item $A \in \central^*$;
        
        \item for all thick sets $H \subseteq S$, the set $A \cap H$ is a set of pointwise recurrence;

        \item for all minimal $S$-systems $X$, $x \in X$, and open $U \subseteq X$ containing $x$, the set $A \cap R(x,U)$ is syndetic.
    \end{enumerate}
\end{theorem}

\begin{proof}
\leanok{}
    These characterizations of sets belonging to $\central^*$ follow from \eqref{eqn_central_star_formula} and the definition and commutativity of the family meet operation in \cref{sec_family_meet}.
\end{proof}

\subsection{Piecewise--, strongly--, and and strongly piecewise-- constructs}
\label{sec_ps_and_strongps_families}

We demonstrate in this subsection how the algebra in \cref{sec_algebra_structure_of_ff} helps us understand and efficiently manipulate the ``piecewise--'', ``strongly--'', ``strongly piecewise--'', and ``very strongly piecewise--'' constructs appearing in ergodic Ramsey theory.  Throughout this section, let $\family$ be a Furstenberg family of subsets of $S$.

A set $A \subseteq S$ is called \emph{piecewise-$\family$} if there exists $B \in \family$ and a thick set $H \subseteq S$ such that $A = B \cap H$, that is,
\[\big\{ A \subseteq S \mid \text{$A$ is piecewise-$\family$}\big\} = \family \classcap \thick.\]
Piecewise syndetic sets -- intersections of syndetic sets and thick sets -- offer a first example. As far as we know, Bergelson, Furstenberg, and Weiss \cite[Thm. II]{bergelson_furstenberg_weiss_2006} were the first to make the construct explicit in this generality; they proved that $\Delta_0^*$ subsets of the positive integers are piecewise \bohrzero{}.
As another example, we see from \cref{thm_central_is_thick_classcap_dcsyndetic} that central sets are precisely the piecewise dynamically central syndetic sets: $\central = \dcsyndetic \classcap \thick$.\\

Dual to the ``piecewise--'' construct is the``strongly--'' construct.
We call a set $A \subseteq S$ \emph{strongly-$\family$} if for all thick sets $H \subseteq S$, the set $A \cap H$ is a member of $\family$, that is,
\[\big\{ A \subseteq S \mid \text{$A$ is strongly-$\family$}\big\} = \syndetic \classcapdual \family.\]
As an example, the family of thickly syndetic sets (defined in \cref{sec_ps_is_partition_reg}) is precisely the family of strongly thick sets: $\PS^* = \syndetic \classcapdual \thick$.  There are two simple but useful observations to make about strongly-$\family$ sets:
\begin{itemize}
    \item since the family meet operation is commutative (\cref{thm_meet_algebra_properties}), we have the following equivalent characterization: a set $A$ is strongly-$\family$ if and only if for all $B \in \family^*$, the set $A \cap B$ is syndetic; and

    \item since $(\family \classcap \thick)^* = \syndetic \classcapdual \family^*$ and $(\family^*)^* = \family$, we see that the family of piecewise-$\family$ sets is dual to the family of strongly-$\family^*$ sets.
\end{itemize}

The family of strongly central sets arises in the work of Bergelson, Hindman, and Strauss \cite{bergelson_hindman_strauss_2012}.
According to our definition above, a set is strongly central if it is central along every thick set.\footnote{A word of caution: Bergelson, Hindman, and Strauss' definition of ``strongly central'' does not match ours here.  They define a set $A \subseteq S$ to be \emph{strongly central} if $\overline{A}$ contains an idempotent ultrafilter in every left ideal of $\beta S$.}
It follows easily from \cref{thm_central_is_thick_classcap_dcsyndetic} that dynamically central syndetic sets are strongly central, that is, 
\begin{align}
    \label{thm_dcsyndetic_is_syndetic_ccdual_central}
    \lean{dcsIsSyndeticMeetCentral}
    \leanok{}
    \dcsyndetic \subseteq \syndetic \classcapdual \central,
\end{align}
but it is not hard to construct examples demonstrating that, in general, the containment is proper.
Bergelson, Hindman, and Strauss \cite[Question 2.3]{bergelson_hindman_strauss_2012} ask whether it is true that a set $A \subseteq S$ is strongly central if and only if all left ideals of $\beta S$ contain an idempotent ultrafilter containing $A$.
Utilizing the duality between ``piecewise--'' and ``strongly--'', we can quickly formulate the equivalent dual form of the question: is it true that a set $A \subseteq S$ is piecewise-$\central^*$ if and only if there exists a left ideal of $\beta S$ every idempotent of which contains $A$? We have no progress to report on this question. \\

Combining the ``piecewise--'' and ``strongly--'' constructs, we call a set $A \subseteq S$ \emph{strongly piecewise-$\family$} if for all thick sets $H \subseteq S$, there exists $B \in \family$ and a thick set $H' \subseteq S$ such that $A \cap H = B \cap H'$, that is,
\[\big\{ A \subseteq S \mid \text{$A$ is strongly piecewise-$\family$}\big\} = \syndetic \classcapdual (\family \classcap \thick).\]
Thus, as the name suggests, ``strongly piecewise--'' is the composition of ``strongly--'' and ``piecewise--''.
We give a characterization of $\central^*$ sets in countable, commutative semigroups in \cref{sec_strong_ip_is_strong_central} using this construct.

A closely-related natural strengthening is the following: we call a set $A \subseteq S$ \emph{very strongly piecewise-$\family$} if for all thick sets $H \subseteq S$, there exists $B \in \family$ and a thick set $H' \subseteq H$ such that $A \cap H' = B \cap H'$. Denoting by $\thick_{\cap H}$ the family of sets $C \subseteq S$ for which $C \cap H \in \thick$, it follows from \cref{lemma_desc_of_very_strong} that
\[\big\{ A \subseteq S \mid \text{$A$ is very strongly piecewise-$\family$}\big\} = \bigcap_{H \in \thick} \big( \thick_{\cap H} \classcap \family \big).\]
It is immediately clear from the definitions that, as the nomenclature suggests, very strongly piecewise-$\family$ sets are strongly piecewise-$\family$.
\cref{mainthm_gen_of_bfw} is conveniently formulated in this language: \emph{$\Delta^*$ sets are very strongly piecewise \bohrzero{}}.
As far as we know, Host and Kra \cite{host_kra_2011} were the first to explicitly define this construct;\footnote{Another word of caution: Host and Kra's definition of ``strongly piecewise-$\family$'' does not match ours here.  Instead, what they define is what we call very strongly piecewise-$\family$.} they proved that $\SG^*_{d}$ subsets of $\Z$ -- a family of combinatorially rich sets generalizing $\Delta^*$ sets defined at the top of \cref{sec_open_directions} -- are very strongly piecewise nil$_d$-\bohrzero{}.
We use this terminology in \cref{thm_strong_piecewise_char_of_central_star} to characterize $\central^*$ sets in countable, commutative semigroups.

In some cases, the families of strongly piecewise-$\family$ sets and very strongly piecewise-$\family$ sets coincide.
We give one such case in \cref{thm_case_when_sp_is_vsp} below, proved with the help of the following lemma which we make separate use of later on.

\begin{lemma}
\label{lemma_thick_conclusion}
\lean{thickIndicatorLemmaInCountCommSemi}
\leanok{}
    Let $S$ be a countable, commutative semigroup, and let $H \subseteq S$ be thick.  There exists a thick set $G \subseteq H$ such that for all $T \subseteq S$ and all left ideals $L \subseteq \beta S$, if $L \subseteq \overline{T}$ and $L \cap \overline{G} \neq \emptyset$, then the set $T \cap H$ is thick.
\end{lemma}

\begin{proof}
\leanok{}
    Enumerate $S$ as $s_1, s_2, \ldots$.  For $n \in \N$, define $F_n = \{s_1, \ldots, s_n\}$.  For all $n \in \N$, there exists $t_n \in S$ such that $\big(F_n \cup (F_n F_n) \big) t_n \subseteq H$.  Define
    \[G = \bigcup_{n=1}^\infty F_n t_n.\]
    That $G$ is thick and that $G \subseteq H$ both follow immediately from its definition.
    We claim that for all $s \in S$, the set $G \setminus s^{-1} H$ is finite.  Indeed, let $s \in S$.  There exists $N \in \N$ such that for all $n \geq N$, $s \in F_n$.  Therefore, for all $n \geq N$, $s F_n t_n \subseteq H$, whereby $F_n t_n \subseteq s^{-1} H$.  It follows that
    \[G \setminus s^{-1} H \subseteq \bigcup_{n = 1}^{N-1} F_n t_n,\]
    which is a finite set.

    Let $T \subseteq S$ and $L \subseteq \beta S$ be a left ideal such that there exists some $p \in L \cap \overline{G}$ and $L \subseteq \overline{T}$.  We must show that $T \cap H$ is thick.  We consider two cases depending on whether $p$ is principal or non-principal.

    Case 1: the ultrafilter $p$ is the principal ultrafilter at some element $s \in S$. Since $p \in L$ and $L$ is a left ideal, we have that $\beta S p \subseteq L \subseteq \overline{T}$.  It follows that $S s \subseteq T$.

    Since $H$ is thick, by \cref{lemma_thick_inverse_shift_intersections}, the set $s^{-1} H$ is thick.  By \cref{lemma_thick_right_shift_is_thick} and the fact that $S$ is commutative, the set $s s^{-1} H \subseteq s S \cap H$ is thick.  But $s S = S s \subseteq T$, and so we see that $T \cap H$ is thick, as was to be shown.

    Case 2: the ultrafilter $p$ is non-principal. Let $F \subseteq S$ be finite. Note that since $p$ is non-principal and $G \in p$ and for all $s \in S$, $G \setminus s^{-1} H$ is finite, we have that for all $s \in S$, $s^{-1} H \in p$.  It follows that
    \[\bigcap_{f \in F}f^{-1} H \in p.\]
    Since $\overline{T} \supseteq L$ and $L$ is a left ideal,
    \[\bigcap_{f \in F} \overline{f^{-1} T} \supseteq L \ni p.\]
    We see that
    \[\bigcap_{f \in F} f^{-1} (T \cap H) = \bigcap_{f \in F}f^{-1} H \cap \bigcap_{f \in F} f^{-1} T \in p,\]
    whereby $\cap_{f \in F} f^{-1} (T \cap H)$ is non-empty, as was to be shown.
\end{proof}

\begin{theorem}
\label{thm_case_when_sp_is_vsp}
\lean{inCommCountSemiStrongPWandVeryStrongPWCoincide}
    Let $S$ be a countable, commutative semigroup, and let $\family$ be a family of subsets of $S$ every member of which is syndetic.  The families of strongly piecewise-$\family$ sets and very strongly piecewise-$\family$ sets coincide:
    \[\syndetic \classcapdual (\family \classcap \thick) = \bigcap_{H \in \thick} \big( \thick_{\cap H} \classcap \family \big).\]
\end{theorem}

\begin{proof}
    It is immediate from the definitions that very strongly piecewise-$\family$ sets are strongly piecewise-$\family$, so we need only show the reverse containment.

    Let $A \subseteq S$ be strongly piecewise-$\family$, and let $H \subseteq S$ be thick.  Let $G \subseteq H$ be a thick set as guaranteed by \cref{lemma_thick_conclusion}.  Since $A$ is strongly piecewise-$\family$, there exists $B \in \family$ and a thick set $T \subseteq S$ such that $A \cap G = B \cap T$.
    Put $B' = B \cup (A \cap T)$, and note that since $\family$ is upward-closed, $B' \in \family$.  We will find a thick set $H' \subseteq H$ such that $A \cap H' = B' \cap H'$.

    Since $T$ is thick, there exists a left ideal $L \subseteq \overline{T}$.  We claim that $L \cap \overline{G} \neq \emptyset$.  Indeed, since $B$ is syndetic, $L \cap \overline{B} \neq \emptyset$.  Then $L \cap \overline{B} \subseteq \overline{T} \cap \overline{B} = \overline{A} \cap \overline{G} \subseteq \overline{G}$.  Therefore, $L \cap \overline{G} \neq \emptyset$.
    We conclude from \cref{lemma_thick_conclusion} that the set $H' \defeq T \cap H$ is thick.
    
    Finally, we claim that
    \[A \cap H' = B' \cap H'.\]
    Expanding the definitions of $H'$ and $B'$, we must show that
    \[A \cap T \cap H = (B \cup (A \cap T)) \cap T \cap H = (B \cap T \cap H) \cup (A \cap T \cap H).\]
    The left  side is clearly a subset of the right.  The right side is a subset of the left since $B \cap T \cap H = A \cap G \cap T \cap H \subseteq A \cap T \cap H$.
\end{proof}

We were not able to determine reasonable necessary and sufficient conditions on a semigroup for the families of strongly piecewise-$\family$ sets and very strongly piecewise-$\family$ sets to coincide.

\section{The local dynamical structure of difference sets}
\label{sec_local_dy_struct}

In this section, we prove Theorems \ref{mainthm_gen_of_bfw} and \ref{mainthm_gen_of_bfw_reformulated}.
The first step (\cref{sec_reduct_to_ur_sets}) is to reduce the problem to a special class of sets arising from minimal dynamics
The second (\cref{sec_diff_and_bz_sets}) is to connect sets of Bohr recurrence and $\Delta$ sets via the regionally proximal relation.

\subsection{Reduction to dynamically syndetic sets}
\label{sec_reduct_to_ur_sets}

The goal in this subsection is \cref{thm_lspwf_upgrade}, which shows that to prove \cref{mainthm_gen_of_bfw_reformulated_stronger}, it suffices to show a related-but-stronger statement for a special class of sets called ``uniformly recurrent'' sets.
This class of sets is defined in the language of symbolic dynamics.

Let $S$ be a (discrete) semigroup, and let $\{0,1\}^S$ be the space of functions $S \to \{0,1\}$ endowed with the product topology.  The semigroup $S$ acts on $\{0,1\}^S$ by continuous maps defined for $s \in S$ and $\omega \in \{0,1\}^S$ by
\[\forall \ t \in S, \ (s \omega)(t) = \omega(ts).\]
The compact Hausdorff space $\{0,1\}^S$ thus becomes an $S$-system. For $s \in S$ and $a \in \{0,1\}$, the cylinder set $[a]_s$ is defined as
\[[a]_s 
\defeq \{\omega \in \{0,1\}^{S} : \omega(s) = a\}.\]

\begin{definition}
    \label{def_ur_set}
    \lean{isURSet}
    \leanok{}
    Let $S$ be a semigroup.  A set $A \subseteq S$ is a \emph{UR set} if $1_A \in \{0,1\}^S$ is $S$-uniformly recurrent.  (Note that the empty set is a UR set.)
\end{definition}

In the following proof, we make use of the monoid extension of a semigroup $S$.
As a set, we define $M_S$ to be $S$ together with a new element.
The set $M_S$ becomes a semigroup by defining the multiplication on it so as to extend the multiplication of $S$ and make the new element the identity element.
Thus, given any semigroup, there exists a semigroup $M_S$ with identity (that is, a monoid) and an injective semigroup homomorphism $i: S \to M_S$.  

\begin{lemma}
    \label{lemma_unif_rec_implies_UR}
    \lean{urSetIsRxU}
    \leanok{}
    Let $S$ be a semigroup, and let $A \subseteq S$.  If $A$ is an UR set, then there exists a minimal $S$-system $X$, a point $x \in X$, and a clopen set $U \subseteq X$ such that $A = R(x,U)$.
\end{lemma}

\begin{proof}
    Let $i: S \to M_S$ be the injective homomorphism into the monoid extension. Via this homomorphism, we can see $\{0,1\}^{M_S}$ as an $S$-system (cf. \cref{sec_hom_systems}).  Moreover, the map $\pi: \{0,1\}^{M_S} \to \{0,1\}^S$ gotten by sending the function $\omega: M_S \to \{0,1\}$ to the function $\omega \circ i: S \to \{0,1\}$ is a factor map of $S$-systems. Indeed, we must show that it is continuous, surjective, and intertwines the actions of $S$: for all $s \in S$, $(s \omega) \circ i = s (\omega \circ i)$.
    
    That $\pi$ is continuous follows from the fact that the preimage of an intersection of cylinder sets is a cylinder set.  That $\pi$ is surjective follows from the fact that $i: S \to M_S$ is injective.  To see that $\pi$ intertwines the actions of $S$, we see that for all $t \in S$,
    \[\big((s \omega) \circ i\big)(t) = (s \omega) (it) = \omega (it is) = \omega (i (ts)) = (\omega \circ i)(ts) = \big(s(\omega \circ i)\big)(t).\]

    By assumption, $1_A \in \{0,1\}^S$ is $S$-uniformly recurrent.  By \cref{thm_ur_lift}, there exists $\omega \in \{0,1\}^{M_S}$ that is $S$-uniformly recurrent and such that $\pi \omega = 1_A$.
    We claim that $A = R_S(\omega, [1]_{e_{M_S}})$, which concludes the proof: we define $X = \overline{S \omega}$, $x = \omega$, and $U = [1]_{e_{M_S}} \cap X$.  Indeed, using that $\omega \circ i = 1_A$, we see $s \in R_S(\omega, [1]_{e_{M_S}})$ if and only if $s\omega \in [1]_{e_{M_S}}$ if and only if $(s \omega)(e_{M_S}) = 1$ if and only if $\omega(is) = 1$ if and only if $1_{A} (s) = 1$ if and only if $s \in A$.
\end{proof}

Our next step is \cref{lem:auslander-ellis}, a take on a famous theorem of Auslander and Ellis which says that every point is a topological dynamical system is proximal to a uniformly recurrent point.
The proof of \cref{lem:auslander-ellis} is best accomplished using ultrafilters, the necessary preliminaries for which are covered in \cref{sec_ultrafilters}.

\begin{lemma}
\label{lem:auslander-ellis}
\lean{syndSetIsUROnThickSet}
\leanok{}
Let $S$ be a semigroup and $A \subseteq S$.  For all thick sets $H \subseteq S$, there exists a UR set $B \subseteq S$ and a thick set $H' \subseteq H$ such that $A \cap H' = B \cap H'$.
\end{lemma}

\begin{proof}
\leanok{}
    Since $H$ is thick, by \cref{lemma_thick_closure_contains_strong_ultrafilter}, there exists a minimal, idempotent ultrafilter $p \in \overline{H} \subseteq \beta S$ such that for all finite $F \subseteq S$, $\cap_{f \in F} f^{-1} H \in p$. Let $1_B = p 1_{A \cap H} \in \{0,1\}^S$.  Since $p$ is minimal we have that $1_B$ is $S$-uniformly recurrent.
    Define $T = \{g \in S \mid 1_B(g) = 1_{A \cap H}(g) \}$.
    
    We claim that the set $H \cap T$ is thick.  Indeed, let $F \subseteq S$ be finite. The set
    \[U \defeq \{(\alpha, \beta) \in (\{0,1\}^S)^2 \mid \forall f \in F, \ \alpha(f) = \beta (f)\} \subseteq (\{0,1\}^S)^2\]
    is, by the definition of the product topology on $\{0,1\}^S$, an open set containing the diagonal.  Since $p = pp$, we see
    \[p 1_B = p p 1_{A \cap H} = p 1_{A \cap H}.\]
    Therefore, we have that $R_S((1_B,1_{A \cap H}), U) \in p$ in the product space $(\{0,1\}^S)^2$ under the diagonal action.  By the property of $p$, we see that $\bigcap_{f \in F} f^{-1} H \in p$.  Since $p$ is a filter,
    \[R_S\big((1_B,1_{A \cap H}), U\big) \cap \bigcap_{f \in F} f^{-1} H \in p.\]
    Choose $g$ in this set.  For all $f \in F$, we have that $fg \in H$ and
    \[1_B(fg) = (g1_B)(f) = (g1_{A \cap H})(f) = 1_{A \cap H}(fg).\]
    Therefore, $Fg \subseteq H \cap T$, as desired.
    
    Set $H' = H \cap T$.  Then $H' \subseteq H$ is thick and $1_B$ and $1_{A \cap H}$ agree on $H'$, that is, $A \cap H' = B \cap H'$, as desired.
\end{proof}

In the following theorem, recall that we denote by $\thick_{\cap H}$ the family of sets $C \subseteq S$ for which $C \cap H \in \thick$.

\begin{theorem}
\label{thm_lspwf_upgrade}
\lean{urContainmentSufficesForFamilyContainment}
\lean{urContainmentSufficesForFamilyContainmentUpgrade}
    Let $S$ be a semigroup, and let $\family$ and $\familytwo$ be families of subsets of $S$. If for all thick $H \subseteq S$, 
    \begin{align}
        \label{eqn_fam_containment_for_ur_sets}
        \big((\thick_{\cap H})^* \classcapdual \family\big) \cap \big\{C \subseteq S : \text{$C$ is a UR set} \big\} \subseteq (\thick_{\cap H})^* \classcapdual \familytwo,
    \end{align}
    then
    \[\bigcup_{H \in \thick} \big( (\thick_{\cap H})^* \classcapdual \family \big) \subseteq \bigcup_{H \in \thick} \big( (\thick_{\cap H})^* \classcapdual \familytwo \big).\]
\end{theorem}

\begin{proof}
    Let $A \in \cup_{H \in \thick} \big( (\thick_{\cap H})^* \classcapdual \family \big)$.  By \cref{lemma_desc_of_very_strong}, there exists a thick set $H \subseteq S$ such that for all thick sets $H' \subseteq H$, the set $A \cap H' \in \family$.

    By \cref{lem:auslander-ellis}, there exists a UR set $B \subseteq S$ and a thick set $T \subseteq H$ such that $A \cap T = B \cap T$. Since $T \subseteq H$, we have that for all thick sets $H' \subseteq T$, $B \cap H' = A \cap H' \in \family$.  Therefore,
    \[B \in \big((\thick_{\cap T})^* \classcapdual \family\big) \cap \big\{C \subseteq S : \text{$C$ is a UR set} \big\}.\]
    By the assumption in \eqref{eqn_fam_containment_for_ur_sets}, we get that $B \in (\thick_{\cap T})^* \classcapdual \familytwo$, that is, by \cref{lemma_desc_of_very_strong}, for all thick sets $H' \subseteq T$, the set $B \cap H' \in \familytwo$.

    To see that $A \in \cup_{H \in \thick} \big( (\thick_{\cap H})^* \classcapdual \familytwo \big)$, we need only note that for all thick subsets $H' \subseteq T$, we have that $A \cap H' \in \familytwo$.  But this holds because $A \cap H' = B \cap H' \in \familytwo$, as desired.
\end{proof}

\subsection{Difference sets and \bohrzero{} sets}
\label{sec_diff_and_bz_sets}

The following theorems provide a connection between the regionally proximal relation, \bohrzero{} sets, and $\Delta$ sets.
More precisely, the first locates $\Delta$-set structure in sets of the form $R(x,V) \cap H$, where $V$ is an open set containing a point that is regionally proximal to $x$ and $H$ is thick, while the second -- utilizing \cref{mainthm_RP_is_eq_relation}, that the regionally proximal relation is the equicontinuous structure relation -- makes the connection to \bohrzero{} sets explicit.
The best precedent for these types of results can be found in \cite[Theorem E]{huang_shao_ye_2016}, which draws a connection between the regionally proximal relation (and its higher order relatives) and the combinatorics of visit time sets in minimal $\Z$-systems.

\begin{theorem}
\label{thm_delta_from_kronecker}
\lean{commVisitTimeSetForRPPairIsDelta}
\leanok{}
    Let $S$ be a commutative semigroup, $X$ be a minimal $S$-system, $(x, y) \in \RP$, and $V \subseteq X$ be an open set containing $y$.  For all thick $H \subseteq S$, the set $R(x,V) \cap H \in \Delta$.
\end{theorem}

\begin{proof}
    Let $H \subseteq S$ be thick. Consider the following propositions as functions of elements $s_1, \ldots, s_N \in S$:
    \begin{align*}
        A(s_1, \ldots, s_N) &: \text{for all $i, j \in \N$ with $1 \leq i < j \leq N$, $s_j \in s_i (R(x,V) \cap H)$,}\\
        B(s_1, \ldots, s_N) &: \text{$\bigcap_{i=1}^N (s_1 \cdots \hat{s_i} \cdots s_N)^{-1} V \neq \emptyset,$}
    \end{align*}
    where, as convenient shorthand in this proof, we write $s_1 \cdots \hat{s_i} \cdots s_N$ to mean the product of all of the $s_j$'s except $s_i$.
    We will prove the following claim.
    
    \begin{clm}
        For all $s_1, \ldots, s_N \in S$ for which $A(s_1, \ldots, s_N)$ and $B(s_1, \ldots, s_N)$ hold, there exists $s_{N+1} \in S$ such that $A(s_1, \ldots, s_{N+1})$ and $B(s_1, \ldots, s_{N+1})$ hold.
    \end{clm}
    
    This suffices to prove that $R(x,V) \cap H \in \Delta$ and finish the proof of the theorem.
    Indeed, assuming the Claim, we must find $s_1, s_2, \ldots \in S$ satisfying: for all $i, j \in \N$ with $i < j$, $s_j \in s_i (R(x,V) \cap H)$.
    To do that, choose $s_1 \in S$ arbitrarily and note that $A(s_1)$ holds vacuously and $B(s_1)$ holds since $V \neq \emptyset$.
    We then build the sequence $s_1, s_2, \ldots$ inductively by repeatedly appealing to the Claim.

    To prove the Claim, suppose $s_1, \ldots, s_N \in S$ are such that $A(s_1, \ldots, s_N)$ and $B(s_1, \ldots, s_N)$ hold.  Define
    \begin{align*}
        U \defeq \bigcap_{i=1}^N (s_1 \cdots \hat{s_i} \cdots s_N)^{-1} V \qquad \text{ and } \qquad G \defeq \bigcap_{i=1}^N (s_1 \cdots \hat{s_i} \cdots s_N)^{-1} H.
    \end{align*}
    Note that $U \subseteq X$ is non-empty and open by $B(s_1, \ldots, s_N)$ and that $G \subseteq S$ is thick by \cref{lemma_preimage_intersection_of_thick_is_thick}.

    By \cref{thm_delta_set_builder_prestep}, the set $R(x,U) \cap R(V,U)$ is syndetic.  The intersection of a thick set and a syndetic set is non-empty, so there exists
    \[t \in R(x,U) \cap R(V,U) \cap G.\]
    Define $s_{N+1} \defeq s_1 \cdots s_N t$.

    By \cref{lemma_return_time_algebra_one} and \cref{lemma_return_to_intersection},
    \begin{align*}
        t \in R(x,U) \cap G &= R \left( x, \bigcap_{i=1}^N (s_1 \cdots \hat{s_i} \cdots s_N)^{-1} V \right) \cap \bigcap_{i=1}^N (s_1 \cdots \hat{s_i} \cdots s_N)^{-1} H \\
        &= \bigcap_{i=1}^N (s_1 \cdots \hat{s_i} \cdots s_N)^{-1} \big( R(x,V) \cap H).
    \end{align*}
    It follows that for all $1 \leq i \leq N$,
    \[s_1 \cdots \hat{s_i} \cdots s_N t \in R(x,V) \cap H.\]
    Multiplying by $s_i$, we see for all $1 \leq i \leq N$, $s_{N+1} \in s_i (R(x,V) \cap H)$.  Combining this with $A(s_1, \ldots, s_N)$ shows that $A(s_1, \ldots, s_{N+1})$ holds.

    To see that $B(s_1, \ldots, s_{N+1})$ holds, we replace $s_{N+1}$ with $s_1 \cdots s_N t$ in
    \begin{align*}
        \bigcap_{i=1}^{N+1} (s_1 \cdots \hat{s_i} \cdots s_N s_{N+1})^{-1} V &= (s_1 \cdots s_N)^{-1} \left( V \cap \bigcap_{i=1}^{N} (s_1 \cdots \hat s_i \cdots s_N t)^{-1} V \right) \\
        &= (s_1 \cdots s_N)^{-1} \left( V \cap t^{-1} U \right).
    \end{align*}
    This is non-empty since $t \in R(V,U)$ and since, by \cref{lemma_minimal_comm_action_implies_surjective}, the element $s_1 \cdots s_N$ acts surjectively on $X$.
\end{proof}

\begin{theorem}
\label{thm_dcthickbohr_is_delta_for_rxu}
\lean{commURSetsOfBohrRecurrenceAreDelta}
\leanok{}
    Let $S$ be a commutative semigroup, and let $B \subseteq S$ be an UR set.  If $B \in \dcthickbohr$, then for all thick $H \subseteq S$, the set $B \cap H \in \Delta$.
\end{theorem}

\begin{proof}
\leanok{}
    By \cref{lemma_unif_rec_implies_UR}, there exists a minimal $S$-system $X$, a point $x \in X$, and a clopen set $U \subseteq X$ such that $B = R(x,U)$.  Assume $B \in \dcthickbohr$. Since sets of Bohr recurrence are non-empty, $B$ is non-empty, and so $U$ is non-empty.  Since $R(x,U) \in \dcthickbohr$, by \cref{lemma_pr_family_concentration}, \cref{lemma_dcthickbohr_is_pr}, and the fact that $U$ is closed, there exists $y \in U$ such that for all non-empty, open $U' \ni y$, the set $R(x,U') \in \dcthickbohr$.  We claim that $(x,y) \in \seq$. To see why, by the definition of $\seq$, we must show that for any equicontinuous $\pi: X \to K$, we have that $\pi x = \pi y$.
    
    Let $\pi: X \to K$ be any equicontinuous factor of $X$.  We will show that $\pi x = \pi y$ by showing that for all open sets $V$ and $W$ containing $\pi x$ and $\pi y$, respectively, $V \cap W \neq \emptyset$.  Let $V$ and $W$ be open sets containing $\pi x$ and $\pi y$, respectively.  Since $K$ is minimal equicontinuous and $V \ni \pi x$, we have by \cref{thm_equiv_char_of_bohr_zero_sets} that $R(\pi x, V) \in \dcsyndeticbohr$. Since $\pi^{-1} W \ni y$ and $R(\pi x, V) \in \dcsyndeticbohr$, we have by \cref{lemma_dcsyndeticbohr_dcthickbohr_dual} that $R(x,\pi^{-1} W) \cap R(\pi x, V) \neq \emptyset$.  Therefore, there exists $s \in S$ such that $sx \in \pi^{-1} W$ and $s \pi x \in V$.  Thus, $s \pi x = \pi s x \in W$ and $s \pi x \in V$, so $V \cap W \neq \emptyset$, as desired.  This completes the proof that $(x,y) \in \seq$.

    We conclude by showing that for all thick $H \subseteq S$, the set $R(x,U) \cap H \in \Delta$. 
    Let $H \subseteq S$ be thick.
    By \cref{mainthm_RP_is_eq_relation}, $\seq = \RP$, so $(x,y) \in \RP$.  Therefore, by \cref{thm_delta_from_kronecker}, $R(x,U) \cap H \in \Delta$, as was to be shown.
\end{proof}

\subsection{Proofs of Theorems  \ref{mainthm_gen_of_bfw} and \ref{mainthm_gen_of_bfw_reformulated}}
\label{sec_proof_of_thms_ab}

Recall that Theorems  \ref{mainthm_gen_of_bfw} and \ref{mainthm_gen_of_bfw_reformulated} were shown to be equivalent in the introduction, so we will prove just \cref{mainthm_gen_of_bfw_reformulated}.

\begin{proof}[Proof of \cref{mainthm_gen_of_bfw_reformulated}]
    By \cref{lemma_desc_of_very_strong}, we see that \cref{mainthm_gen_of_bfw_reformulated} asserts that,
    \begin{align}
        \label{eqn_thmB_goal}
        \bigcup_{H \in \thick} \big( (\thick_{\cap H})^* \classcapdual \dcthickbohr \big) \subseteq \Delta,
    \end{align}
    where $\thick_{\cap H}$ denotes the family of sets $A \subseteq S$ for which $A \cap H \in \thick$.

    To show this, we will apply \cref{thm_lspwf_upgrade} to the families $\family \defeq \dcthickbohr$ and $\familytwo \defeq \Delta$.  To check the condition in \eqref{eqn_fam_containment_for_ur_sets}, let $H \subseteq S$ be thick.  Note the following.
    \begin{enumerate}
        \item Since $H \in \thick_{\cap H}$, we see that $\thick_{\cap H} \neq \emptyset$, whereby $(\thick_{\cap H})^* \neq \powerset(S)$.  By \cref{thm_familymeet_containment_conditions},
        \[(\thick_{\cap H})^* \classcapdual \dcthickbohr \subseteq \dcthickbohr.\]

        \item \cref{thm_dcthickbohr_is_delta_for_rxu} gives that
        \[\dcthickbohr \cap \big\{C \subseteq S : \text{$C$ is a UR set} \big\} \subseteq \syndetic \classcapdual \Delta.\]

        \item     Since $\syndetic = \thick^* \subseteq (\thick_{\cap H})^*$, by monotonicity of family meet,
        \[\syndetic \classcapdual \Delta \subseteq (\thick_{\cap H})^* \classcapdual \Delta.\]
    \end{enumerate}
    Combining 1., 2., and 3., we see that
    \begin{align*}
        \big((\thick_{\cap H})^* \classcapdual \dcthickbohr\big) \cap \big\{C \subseteq S : \text{$C$ is a UR set} \big\} &\subseteq \dcthickbohr \cap \big\{C \subseteq S : \text{$C$ is a UR set} \big\} \\
        &\subseteq \syndetic \classcapdual \Delta \\
        &\subseteq (\thick_{\cap H})^* \classcapdual \Delta,
    \end{align*}
    verifying the condition in \eqref{eqn_fam_containment_for_ur_sets}.

    By \cref{thm_lspwf_upgrade}, we conclude that 
    \[\bigcup_{H \in \thick} \big( (\thick_{\cap H})^* \classcapdual \dcthickbohr \big) \subseteq \bigcup_{H \in \thick} \big( (\thick_{\cap H})^* \classcapdual \Delta \big).\]
    For all $H \in \thick$, we by have that $(\thick_{\cap H})^* \classcapdual \Delta \subseteq \Delta$ by \cref{thm_familymeet_containment_conditions} just as in 1. above.  The containment in \eqref{eqn_thmB_goal} follows immediately from the previous line, proving \cref{mainthm_gen_of_bfw_reformulated}.
\end{proof}

\subsection{Proof of \cref{mainthm_local_ipstar_in_cstar}}
\label{sec_strong_ip_is_strong_central}

A set $A \subseteq S$ is called an \emph{$\IP$ set} if there exists a sequence $s_1, s_2, \ldots \in S$ such that for all finite $F \subseteq \N$, the product $\Pi_{f \in F} s_f \in A$, where the product is taken in increasing order of index.  We denote by $\IP$ the family of $\IP$ subsets of $S$.

It is well-known that a set is $\IP$ if and only if it is contained in an idempotent ultrafilter on $S$ \cite[Thm. 16.4]{hindman_strauss_book_2012}.  Consequently, we see immediately from the ultrafilter characterization of central sets that central sets are $
\IP$ sets, that is, $\central \subseteq \IP$.  It follows that $\IP^* \subseteq \central^*$, and for most semigroups this inclusion is strict.  Nevertheless, \cref{mainthm_local_ipstar_in_cstar} shows that $\central^*$ sets possess some local $\IP^*$ structure, at least in countable, commutative semigroups.

\begin{proof}[Proof of \cref{mainthm_local_ipstar_in_cstar}]
    By \cref{lemma_desc_of_very_strong}, we see that \cref{mainthm_local_ipstar_in_cstar} asserts that
    \begin{align*}
        \central^* \subseteq 
        \bigcap_{H \in \thick} \big( \thick_{\cap H} \classcap \IP^* \big),
    \end{align*}
    where $\thick_{\cap H}$ denotes the family of sets $A \subseteq S$ for which $A \cap H \in \thick$.
    Taking the dual and using the DeMorgan properties, we see that this is equivalent to
    \begin{align}
        \label{eqn_thmC_goal}
        \bigcup_{H \in \thick} \big( (\thick_{\cap H})^* \classcapdual \IP \big) \subseteq \central.
    \end{align}

To show this, we apply \cref{thm_lspwf_upgrade} to the families $\family \defeq \IP$ and $\familytwo \defeq \central$ exactly as we did in the proof of \cref{mainthm_gen_of_bfw_reformulated}. Following that proof exactly as it is written, we need only verify the condition in \eqref{eqn_fam_containment_for_ur_sets}: for all thick sets $H \subseteq S$,
\begin{align*}
    \big((\thick_{\cap H})^* \classcapdual \IP\big) \cap \big\{C \subseteq S : \text{$C$ is a UR set} \big\} \subseteq (\thick_{\cap H})^* \classcapdual \central.
\end{align*}

Let $H \subseteq S$ be thick, and let $B \in (\thick_{\cap H})^* \classcapdual \IP$ be a UR set. By \cref{lemma_unif_rec_implies_UR}, there exists a minimal $S$-system $X$, a point $x \in X$, and a clopen set $U \subseteq X$ such that $B = R(x,U)$.  We have that for all thick $H' \subseteq H$, the set $R(x,U) \cap H' \in \IP$, and we must show that for all thick $H' \subseteq H$, the set $R(x,U) \cap H' \in \central$.

    Let $H' \subseteq H$ be thick. By \cref{lemma_thick_conclusion}, there exists a thick set $G \subseteq H'$ with the stipulated property.
    By assumption, the set $R(x,U) \cap G$ is $\IP$, so it is contained in an idempotent ultrafilter $p$.
    The point $px \in U$ is proximal to $x$.  By \cref{lemma_uniform_near_diag}, there is a non-empty, open set $\alpha \subseteq X^2$ containing the diagonal and an open set $V \subseteq X$ containing $p x$ such that
    \begin{align}
        \label{eqn_centered_return_expression}
        R\big((x,p x), \alpha\big) \cap R(p x,V) \subseteq R(x,U).
    \end{align}
    Define
    \[L = \{p\} \cup \beta S p \subseteq \beta S,\]
    and note that $L$ is a left ideal.
    Since $p \in \overline{G}$, we see that $p \in L \cap \overline{G}$.
    Also, we have that $L \subseteq \overline{R\big((x,p x), \alpha\big)}$.  Indeed, $p \in \overline{R\big((x,p x), \alpha\big)}$ since $p(x,px) = (px,px) \in \alpha$, and for all $q \in \beta S$, we have that $qp \in \overline{R\big((x,p x), \alpha\big)}$ for the same reason.

    Intersecting both sides of the containment in \eqref{eqn_centered_return_expression} with the set $H'$, we see
    \[R\big((x,p x), \alpha\big) \cap H' \cap R(p x,V) \subseteq R(x,U) \cap H'.\]
    By \cref{lemma_thick_conclusion}, the set $H' \cap R\big((x,px),\alpha\big)$ is thick.
    Since $px \in V$ is $S$-uniformly recurrent, we have that $R(p x,V) \in \dcsyndetic$.
    By \cref{thm_central_is_thick_classcap_dcsyndetic}, we see that $R(x,U) \cap H'$ is central, as was to be shown.
\end{proof}

\subsection{Strong family coincidences}

We end this section by deriving from Theorems \ref{mainthm_gen_of_bfw}, \ref{mainthm_gen_of_bfw_reformulated}, and \ref{mainthm_local_ipstar_in_cstar} some curious coincidences among ``strong--'' families, using the terminology from \cref{sec_ps_and_strongps_families}.  The following elementary lemma helps with these deductions.

\begin{lemma}
    \label{lemma_strong_family_deduction_helper}
    \lean{strongFamilyDeductionHelper}
    Let $S$ be a semigroup.  For all families $\family$,
    \[\syndetic \classcapdual \bigg ( \bigcup_{H \in \thick} \big( (\thick_{\cap H})^* \classcapdual \family \big) \bigg) = \syndetic \classcapdual \family.\]
\end{lemma}

\begin{proof}
    First we will show that the left hand side is contained in the right.  Let $H \in \thick$.  Since $\thick_{\cap H} \neq \emptyset$, $(\thick_{\cap H})^* \neq \powerset(S)$, so by \cref{thm_familymeet_containment_conditions}, $(\thick_{\cap H})^* \classcapdual \family \subseteq \family.$  Since $H \in \thick$ was arbitrary, the containment follows by the monotonicity of family meet.

    To see the reverse inclusion, suppose $A \in \syndetic \classcapdual \family$, and let $T \subseteq S$ be thick.  We will show that $A \cap T \in (\thick_{\cap T})^* \classcapdual \family$.  Let $H \subseteq T$ be thick.  Since $H$ is thick, the set $A \cap T \cap H = A \cap H \in \family$, as was to be shown.
\end{proof}

\begin{theorem}
\label{mainthm_gen_of_bfw_reformulated_stronger}
\lean{commSyndFamMeetDeltaIsSnydFamMeetSetOfBohrRecurrence}
\lean{commSyndFamMeetDeltaIsSnydFamMeetDeltaZero}
\leanok{}
    Let $S$ be a commutative semigroup.  The families of strongly-$\Delta$, strongly-$\Delta_0$, and strongly-$\dcthickbohr$ sets coincide:
    \[\syndetic \classcapdual \Delta = \syndetic \classcapdual \Delta_0 = \syndetic \classcapdual \dcthickbohr.\]
    Equivalently, the families of piecewise-$\Delta^*$, piecewise-$\Delta_0^*$, and piecewise-\bohrzero{} coincide:
    \[\thick \classcap \Delta^* = \thick \classcap \Delta_0^* = \thick \classcap \dcsyndeticbohr.\]
\end{theorem}

\begin{proof}
\leanok{}
    That
    \[\syndetic \classcapdual \Delta \subseteq \syndetic \classcapdual \Delta_0 \subseteq \syndetic \classcapdual \dcthickbohr\]
    follows from $\Delta \subseteq \Delta_0$ (this is immediate from the definitions), $\Delta_0 \subseteq \dcthickbohr$ (\cref{thm_delta_zero_set_of_bohr_rec}), and the monotonicity of $\classcapdual$ (\cref{thm_classcapdual_is_monotonic}). So we have only left to show that
    \[\syndetic \classcapdual \dcthickbohr \subseteq \syndetic \classcapdual \Delta.\]
    By the monotonicity of family meet and \cref{lemma_strong_family_deduction_helper}, it suffices to show that
    \[ \bigcup_{H \in \thick} \big( (\thick_{\cap H})^* \classcapdual \dcthickbohr \big) \subseteq \Delta.\]
    But this is precisely the content of \cref{mainthm_gen_of_bfw_reformulated}.
\end{proof}

\begin{theorem}
\label{thm_strong_piecewise_char_of_central_star}
\lean{strongIPIffStrongCentralInCountCommSemi}
\lean{cStarIsStronglyPiecewiseIPStarAndCStar}
\leanok{}
    Let $S$ be a countable, commutative semigroup. The families of strongly central and strongly-$\IP$ sets coincide:
    \begin{align}
        \label{eqn_strong_central_is_strong_ip}
        \syndetic \classcapdual \central = \syndetic \classcapdual \IP.
    \end{align}
    Equivalently, the families of piecewise-$\central^*$ and piecewise-$\IP^*$ sets coincide:
    \[\thick \classcap \central^* = \thick \classcap \IP^*.\]
    Furthermore, for all $A \subseteq S$, the following are equivalent:
    \begin{enumerate}
        \item $A \in \central^*$;

        \item the set $A$ is strongly piecewise $\IP^*$, that is, $A \in \syndetic \classcapdual (\IP^* \classcap \thick)$;
        
        \item the set $A$ is strongly piecewise $\central^*$, that is, $A \in \syndetic \classcapdual (\central^* \classcap \thick)$.
    \end{enumerate}
    In other words,
    \begin{align}
    \label{eqn_central_star_char}
    \central^* = \syndetic \classcapdual (\IP^* \classcap \thick) = \syndetic \classcapdual (\central^* \classcap \thick).
    \end{align}
\end{theorem}

\begin{proof}
\leanok{}
    Since $\central \subseteq \IP$, by the monotonicity of family meet, $\syndetic \classcapdual \central \subseteq \syndetic \classcapdual \IP$. To see the reverse inclusion, by the monotonicity of family meet and \cref{lemma_strong_family_deduction_helper}, it suffices to show that
    \[ \bigcup_{H \in \thick} \big( (\thick_{\cap H})^* \classcapdual \IP \big) \subseteq \central.\]
    But this is precisely the content of \cref{mainthm_local_ipstar_in_cstar}.

    We move on now to show \eqref{eqn_central_star_char}.  Taking the dual in \eqref{eqn_strong_central_is_strong_ip}, we see that $\thick \classcap \central^* = \thick \classcap \IP^*$, whereby $\syndetic \classcapdual (\central^* \classcap \thick) = \syndetic \classcapdual (\IP^* \classcap \thick)$.  The inclusion
    \[\central^* \subseteq \syndetic \classcapdual (\central^* \classcap \thick)\]
    follows immediately from the definitions, leaving us only to show that
    \[\syndetic \classcapdual (\central^* \classcap \thick) \subseteq \central^*.\]
    Combining \cref{thm_central_is_thick_classcap_dcsyndetic}, \eqref{thm_dcsyndetic_is_syndetic_ccdual_central}, and monotonicity, we see
    \[\central = \thick \classcap \dcsyndetic \subseteq \thick \classcap (\central \classcapdual \syndetic).\]
    By the DeMorgan laws and the fact that the dual is antitone, taking the dual, we see
    \[\syndetic \classcapdual (\central^* \classcap \thick) = \big(\thick \classcap ( \central \classcapdual \syndetic )\big)^* \subseteq \central^*,\]
    as was to be shown.
\end{proof}

\section{Open directions}
\label{sec_open_directions}

In this section, we collect the most salient open directions suggested by this work.

We make repeated reference to $\IP_0$ sets and $\SG_d$ in the sections below, so we define them here.  Let $S$ be a semigroup and $A \subseteq S$.  The set $A$ is an $\IP_0$ set if for all $N \in \N$, there exist $s_1, \ldots, s_N \in S$ such that for all $1 \leq k \leq N$ and all integers $1 \leq n_1 < n_2 < \cdots < n_k \leq N$, the product $\Pi_{i = 1}^k s_{n_i} \in A$, where the product is taken in increasing order of index.
The set $A$ is an \emph{$\SG_d$-set} (cf. \cite{host_kra_2011}), $d \in \N$, if there exists a sequence $s_1, s_2, \ldots \in S$ such that for all $k \in \N$ and all positive integers $n_1 < n_2 < \cdots < n_k$ with $n_j - n_i \leq d$, the product $\Pi_{i = 1}^k s_{n_i} \in A$, where the product is taken in increasing order of index.  
We denote the families of $\IP_0$- and $\SG_d$-sets by $\IP_0$ and $\SG_d$, respectively.

\subsection{Higher-order analogues}
\label{sec_higher_order_analogues}

Our main result -- that $\Delta^*$ subsets of commutative semigroups are very strongly piecewise \bohrzero{} -- can be understood as part of a broader program of exploring the kinds of local dynamical structure enjoyed by combinatorially-defined sets.
With the right inputs, the approach we take in this paper should be flexible enough to similarly treat a number of other combinatorially-defined families.

Let us demonstrate by explaining how to recover the main result of Host and Kra in \cite{host_kra_2011}.
We will need the higher-order regionally proximal relations $\RP^{[d]}$, $\SG_d$-sets, and a class of dynamically central syndetic sets called nil$_d$-Bohr$_0$-sets.
The precise definitions of these objects can be found in \cite{host_kra_2011} and \cite{host_kra_maass_2010}; it suffices to know for the current discussion only that $\RP^{[d]}$ generalizes $\RP$, $\SG_d$-sets generalize $\Delta$-sets, and nil$_d$-Bohr$_0$-sets generalize \bohrzero{} sets.

Host and Kra proved that $\SG_{d}^*$ subsets of $\Z$ are very strongly piecewise nil$_d$-Bohr$_0$.
By \cref{thm_lspwf_upgrade}, it suffices to show, roughly, that along any thick set, a visit time set $R(x,U)$ that is a set of nil$_d$-Bohr$_0$ recurrence along any sub-thick set is $\SG_{d}$ along that sub-thick set.
We can derive that from \cref{thm_dcthickbohr_is_delta_for_rxu}, provided we have the proper analogue of \cref{thm_delta_from_kronecker}: we need to know that a set of the form $R(x,V)$ where $V$ contains a point $y$ for which $(x,y) \in \RP^{[d]}$ is an $\SG_{d}$-set along every thick set.
Huang, Shao, and Ye \cite[Thm 7.3.2]{huang_shao_ye_2016} show this but without the thick set; we strongly suspect that a closer inspection of their proof would yield the necessary result.

The same outline -- together with the same improvement on a result of Huang, Shao, and Ye -- should be able to show that $\IP_0^*$ sets are locally nil$_\infty$-Bohr$_0$.
(Nil$_\infty$-Bohr$_0$ are the class of dynamically central syndetic sets arising from infinite-step nilsystems.)
Indeed, note that \cite[Prop. 7.3.1]{huang_shao_ye_2016} says that $\family_{fSG_{\infty}} = \family_{fip}$, which we denote by $\IP_0^*$.
To derive the correct analogue of \cref{thm_delta_from_kronecker}, one would only need to introduce a thick set into \cite[Thm. 7.5.1]{huang_shao_ye_2016}.\\

Our attention so far in this subsection has been on $\Z$-actions, but the higher-order regionally proximal relations make sense and have been studied more generally.
Introduced by Host, Kra, and Maass \cite{host_kra_maass_2010}, these relations, $\RP^{[d]}$ for $d \geq 1$, were shown by Shao and Ye \cite{shao_ye_2012} to be equivalence relations for minimal actions of abelian groups.
We showed in \cite{blahodatna_detmold_glasscock_le_RP_2026} that $\RP$ is an equivalence relation for minimal actions of commutative semigroups.
If the same is true for the higher-order analogues -- we define $\RP^{[d]}$ for semigroup actions and ask this question explicitly in \cite{blahodatna_detmold_glasscock_le_RP_2026} -- it would be a natural next step to seek analogues of the results discussed above.

There are at least two immediate obstacles when seeking generalizations of those results from $\Z$ to arbitrary commutative semigroups.
The first is a nice description of the factor systems $X / \RP^{[d]}$ for $d \geq 2$.
In the case of $\Z$-actions, it is well-understood now that these factors have the structure of an inverse limit of rotations on $d$-step nilmanifolds.
For actions of more general groups and semigroups, the picture is more complicated.
We have nothing new to say in this direction but point the interested reader to \cite{glasner_gutman_ye_2018,gutman_manners_varju_2020} for more information.

The second obstacle is to divine the correct combinatorial structure that generalizes $\Delta$-sets in a result like \cref{thm_delta_from_kronecker}.
We saw above the connection between $\RP^{[d]}$ and $\SG_d$-sets, and there is no apparent obstruction to defining $\SG_d$-sets for arbitrary semigroups (as we did at the top of this section).
We are led to ask the following question in attempting to understand whether this is the ``correct'' combinatorial structure.

\begin{question}
\label{quest_rp_sg_connection}
    Let $S$ be a commutative semigroup, $X$ be a minimal $S$-system, $d \in \N$, $(x,y) \in \RP^{[d]}$, and $V \subseteq X$ be an open set containing $y$.  Is it true that for all thick $H \subseteq S$, the set $R(x,V) \cap H$ is an $\SG_d$-set?
\end{question}

\cref{thm_delta_from_kronecker} gives a positive answer to \cref{quest_rp_sg_connection} in the case that $d=1$.  Key to \cref{thm_delta_from_kronecker} is not the definition of the $\RP^{[1]}$-relation but the strengthening of it given in  \cref{thm_delta_set_builder_prestep}.  Precedent for a higher-order analogue of \cref{thm_delta_set_builder_prestep} can be found in \cite{qiu_ye_2025}, where a ``higher-order Veech's relation'' is shown for all minimal actions of abelian groups.

\subsection{Strong family coincidences}

In a typical semigroup, the class of $\IP$ sets is far smaller than the class of central sets.  Indeed, central sets are piecewise syndetic, while $\IP$ sets need not be.
It was surprising to us to discover, then, that at least in countable, commutative semigroups, strongly-$\IP$ sets are strongly central (\cref{thm_strong_piecewise_char_of_central_star}).
We suspect this could be true in greater generality, but we were unable to show it.

\begin{question}
    Let $S$ be a semigroup and $A \subseteq S$.  If for all thick sets $H \subseteq S$, the set $A \cap H$ is an $\IP$ set, is it true that for all thick sets $H \subseteq S$, the set $A \cap H$ is central?  That is, is it the case that
    \[\syndetic \classcapdual \IP \subseteq \syndetic \classcapdual \central?\]
    (Note that the reverse containment always holds since $\central \subseteq \IP$.)
\end{question}

That strongly-$\IP$ sets are strongly central is not the only unexpected coincidence involving the strongly-- construct in this paper.
Indeed, we showed in \cref{mainthm_gen_of_bfw_reformulated_stronger} that
\[\syndetic \classcapdual \Delta_0 = \syndetic \classcapdual \Delta.\]
Thus, if a set is $\Delta_0$ along all thick sets, then it is a $\Delta$ set.
This is a purely combinatorial statement; it would be interesting to see a purely combinatorial proof.
It would also be interesting to know whether a similar statement holds for $\IP_0$ sets.

\begin{question}
\label{quest_strongly_ip_vs_ipzero}
    Let $S$ be a semigroup and $A \subseteq S$.  If for all thick sets $H \subseteq S$, the set $A \cap H$ is an $\IP_0$ set, is it true that for all thick sets $H \subseteq S$, the set $A \cap H$ an $\IP$ set?  That is, is it the case that
    \[\syndetic \classcapdual \IP_0 \subseteq \syndetic \classcapdual \IP?\]
    (Note that the reverse containment always holds since $\IP \subseteq \IP_0$.)
\end{question}

It seems pertinent to mention that the relationship between the families $\Delta$ and $\Delta_0$ is, in one important dynamical way, not analogous to the relationship between the families $\IP$ and $\IP_0$.  Indeed, following a remark in \cite{bergelson_leibman_2018}, while both $\Delta^*$- and $\Delta_0^*$-recurrence characterize group rotations, $\IP^*$-recurrence characterizes distal systems while $\IP_0^*$-recurrence characterizes infinite-step pro-nilsystems.
This leads us to believe that there might be a negative answer to \cref{quest_strongly_ip_vs_ipzero}.

\subsection{A confluence of algebraic operations on families}
\label{sec_family_meet_with_algebra}

Let $(S,\cdot)$ be a semigroup.  That the binary operation on $S$ extends to an associative binary operation on $\beta S$ -- defined explicitly in \eqref{eqn_ultra_mult} -- is a well-known fact with many applications. Recently, the third and fourth authors \cite{glasscock_le_2025,glasscock_le_synd_idem_filters_2025} found it useful to extend that operation to arbitrary families.  Thus, if $\family$ and $\familytwo$ are Furstenberg families on $S$, we define
\begin{align}
    \label{eqn_fam_product}
    \family \cdot \familytwo \defeq \big\{ A \subseteq S \mid \big\{ s \in S \mid \{ t \in T \mid st \in A \} \in \familytwo \big\} \in \family \big\}.
\end{align}
It is quick to check that this product yields an associative, monotone operation on the set of families of $S$ making $(\families(S),\cdot)$ a semigroup containing $\beta S$ as a subsemigroup.

We saw a number of curious and useful identities in \cref{sec_algebra_structure_of_ff} between the family dual, meet, and join operations.  Together with the product in \eqref{eqn_fam_product}, we are led naturally to ask: what kinds of sentences involving $\ast$, $\classcapdual$, $\classcap$, and $\cdot$ hold true?  Here are two examples for which we have found applications:
\begin{itemize}
    \item (cf. \cite[Lemma 2.1 (iv)]{glasscock_le_synd_idem_filters_2025}) $(\family \cdot \familytwo)^* = \family^* \cdot \familytwo^*$, and

    \item (cf. \cite[Lemma 2.3 (ix)]{glasscock_le_2025}) $(\family_1 \cdot \family_2) \classcap (\familytwo_1 \cdot \familytwo_2) \subseteq (\family_1 \classcap \familytwo_1) \cdot (\family_2 \classcap \familytwo_2)$.
\end{itemize}
We leave the proofs of these as exercises for the curious reader.
Beyond listing more of these identities, it would be interesting to understand the relationships between these operations more systematically.

\bibliographystyle{abbrv}
\bibliography{bib}

\bigskip
\bigskip
\footnotesize
\noindent
Angelina Blahodatna\\
\textsc{University of Massachusetts Lowell}\par\nopagebreak
\noindent
\href{mailto:angelina_blahodatna@uml.edu}
{\texttt{angelina{\_}blahodatna@uml.edu}}

\bigskip
\noindent
Lauren Detmold\\
\textsc{University of Massachusetts Lowell}\par\nopagebreak
\noindent
\href{mailto:lauren_detmold@uml.edu}
{\texttt{lauren{\_}detmold@uml.edu}}

\bigskip
\noindent
Daniel Glasscock\\
\textsc{University of Massachusetts Lowell}\par\nopagebreak
\noindent
\href{mailto:daniel_glasscock@uml.edu}
{\texttt{daniel{\_}glasscock@uml.edu}}

\bigskip
\noindent
Anh N. Le\\
\textsc{University of Denver}\par\nopagebreak
\noindent
\href{mailto:anh.n.le@du.edu}
{\texttt{anh.n.le@du.edu}}

\end{document}